\documentclass[10pt,reqno]{amsart}
\usepackage{amsmath,amssymb,amsthm}
\usepackage{mathrsfs}
\usepackage{graphics}
\usepackage{graphicx,color, xcolor, soul} 
\usepackage{epsfig}

\theoremstyle{plain}
\newtheorem{thm}{Theorem}[section]
\newtheorem{lemma}{\textbf{Lemma}}[section]

\newtheorem{prop}[thm]{Proposition}

\theoremstyle{definition}
\newtheorem{defn}{Definition}[section]

\newtheorem{rem}{Remark}[section]
\newtheorem{example}{Example}[section]

\numberwithin{equation}{section}
\begin{document}
	\begin{center}
			{\bf{On Abstract Nonlinear Integro-Dynamic Equations in Time Scale}}
		\vspace{.5cm}
		
		Abdul Awal Hadi Ahmed$^1$ and Bipan Hazarika$^{2,\ast}$ 
		
		\vspace{.2cm}
		\footnotesize  $^{1,2}$Department of Mathematics, Gauhati University, Guwahati 781014, Assam, India\\
		
		\vspace{.2cm}
		Email: $^{1}$aawal8486@gmail.com;   
		$^2$bh\_rgu@yahoo.co.in; bh\_gu@gauhati.ac.in
	\end{center}
	\title{}
	\author{}
	\thanks{{\today}, $^\ast$The corresponding author}

	\begin{abstract} In this paper, we investigate the existence of the asymptotically almost automorphic solution of the following type of abstract nonlinear integro-dynamic equation
		\begin{eqnarray*}
			y^{\Delta}(s) &=&Ay(s)+\mathcal{F}\left(s,y(s),\int\limits_{t_0}^{s}{\mathcal{H}(s,\tau,y(\tau))}\Delta\tau\right),~ s\in\mathbb{T}^k,\\
			y(0)&=&y_0
		\end{eqnarray*} 
		in the Banach space of continuous function on a time scale $\mathbb{T}$. We apply the Krasnoselskii fixed point theorem to show the existence of an almost automorphic solution of the above dynamic equation.
		AMS Subject Classification No: 26A24; 26A33; 26E70; 34B15; 34N05; 39A10.
	\end{abstract}
	\maketitle
	\section{Introduction}
	In 1988, German Mathematician Stefan Hilger introduced the concept of time scale in Mathematics through his Ph.D. thesis. After that, Hilger published two interesting articles on this topic \cite{HA,HD}. Time scale unifies the discrete and continuous calculi, to study them simultaneously rather than separately. For details on time scale calculus,  see the monographs \cite{Bohner2,Bohner1}. In recent times researchers have been quite actively working on dynamic equations to merge results from both differential and difference equations. Dynamic equations are associated with several real-world phenomena involving discrete as well as continuous variables, for example, we refer to Population Dynamics \cite{Kaufmann,Zhuang}, Optimization Theory \cite{Zhu}, Economics \cite{Atici}, production-inventory modelling \cite{Uysal}, etc. 
	
	Periodic functions has a wide range of applications in real word problem, for example in field of astronomy, physics, biology, in the use of signal processing, control system, electrical engeneering etc. But to be more precise we actually come accross phenomena that involves disturbed form of periodic functions. to generalise this kind of issue, we make use of functions which do not have a exact periodic nature but having several periodic nature. We call them almost periodic. Again we generalise them to almost automorphic(AA) functions and then asymptically almost automorphic(AAA) functions, etc.
	
	In 2011, Y. Li and C. Wang, \cite{Li and Wang} introduced the concept of almost periodic functions on time scales. Then they applied the results to a class of high-order Hopfield neural networks with variable delays.  
	On the other hand, in 2012 , Hamza and Oraby \cite{Hamza and Oraby}, studied the stability of abstract dynamic equation $x^\Delta(t)=Ax(t),$ where $A$ is the infinitesimal generator of a $C_0-$semigroup $T=\left\{T(s):s\in\mathbb{T}\subset\mathcal{L}(\mathbb{Y})\right\}.$
	
	In 2013, Lizama and Mesquita, \cite{Lizama and Mesquita} introduced the concept of almost automorphic functions on time scale and presented the first basic results concerning such functions. They also studied the existence and uniqueness of solutions of the following dynamic equation $$x^\Delta(t)=A(t)x(t)+f(t,x(t)), ~t\in \mathbb{T},$$ over Euclidian space $\mathbb{R}^n.$ Assuming the equation $x^\Delta(t)=A(t)x(t)$ admitting exponential Dichotomy and $f:\mathbb{T}\times X\to X$ satisfying global Lipschitz condition on its second variable, they proved the existence and uniqueness of the almost automorphic solution to the above dynamic equation. 
	
	In 2014, Gu\'er\'ekata et al. \cite{Guerekata 1} presented almost automorphic functions of order $n.$ They also studied the existence and uniqueness, global stability of the solution of first order dynamic equation with finite time varying delay.
	
Using their results of \cite{Guerekata 1,Hamza and Oraby,Lizama and Mesquita,Li and Wang}, in 2015, Milc\'e and Mado proved the existence and uniqueness of almost automorphic mild solution of the following dynamic equation $$u^{\Delta}(t)= A(t)u(t)+f\left(t,u(t),\int\limits_{0}^{t}{\phi(s,u(s))}\Delta s\right),~ t\in\mathbb{T}^k,$$ on Euclidian space, $\mathbb{R}^n.$ $A:\mathbb{T}\to\mathbb{R}^n$ is a matrix-valued function. $\phi,f$ both satisfy some Lipschitz conditions. 
	After that in 2015, Gu\'er\'ekata et al. studied semilinear dynamic equation$x^\Delta(t)=Ax(t)+f(t,x(t))$ in Banach space, where $A$ is the infinitesimal generator of a $C_0-$semigroup of bounded linear operators. They established an almost automorphic mild solution to the equation. 
	In 2018, Hamza and Oraby \cite{Hamza and Oraby 1} studied the stability of nonlinear dynamic equation
	\begin{eqnarray*}
		x^\Delta(t)&=&A(t)x(t)+f(t,x(t)),~ t\in[\tau,\infty)_\mathbb{T}.
	\end{eqnarray*} Several other aspects are also studied.

	The notion of asymptotically almost automorphic functions also abbreviated as AAA functions, are typically those type of functions which behaves like almost automorphic functions after a certain initial transient phase. This type of functions behaves differently at the initial phase and gradually settles down with the behaviour of an equivalent almost automorphic function. These type of functions are useful for studying long-term behaviour of a dynamical system,specially used while dealing with a non-autonomous system or systems which are influenced bt some external influences.
	In 2018, Guerekata et al. \cite{Cao} studied the existence of asymptotically almost automorphic mild solutions for nonautonomous semilinear evolution equations. 
	In 2019, Lizama and Mesquita, \cite{Lizama and Mesquita 1} introduced the concept of asymptotically almost automorphic functions of order $n$ on time scale. They established fundamental properties of such functions and investigated the unique solution of IVP associated to the semilinear equation, \begin{eqnarray*}
		x^\Delta(t)&=&A(t)x(t)+f(t,x(t)),~ t\in[t_0,\infty)_\mathbb{T}\\
		x(t_0)&=&x_0
	\end{eqnarray*}
    Following the above results, in 2016, Milce \cite{Milce} studied the existence of asymptotically almost automorphic solutions for the following integro-dynamic equation $$x^\Delta(t)=Ax(t)+\int\limits_{0}^{t}{B(t-s)x(s)\Delta s}+f(t,x(t))$$ with nonlocal initial condition $x(0)=x_0+\psi(x).$ $A$ is matrix, $B$ is matrix-valued function, $f,\psi$ are rd-continuous functions satisfying some kind Lipschitz condition.
    
	In 2022, Bohner et al. \cite{Bohner 3}, gave some qualitative results for nonlinear integro-dynamic equations via integral inequalities. They considered the equation 
	\begin{eqnarray*}
		x^{\Delta}(t)+p(t)x^{\sigma}(t)&=&\mathcal{F}\left(t,x(t),\int\limits_{t_0}^{t}{\mathcal{H}(t,s,x(x))}\Delta s\right),~ t\in\mathbb{T}^k,  a\leq t\leq b\\
		x(t_0)&=&x_0\in \mathbb{R}^n
	\end{eqnarray*}  and existence, stability, boundedness, and dependence of the solution on initial data are discussed. 

	In 2023, Cosme et al. \cite{Cosme}, discussed the existence and stability of the bounded solution of the following abstract dynamic equation 
	\begin{eqnarray*}
		z^{\Delta}(t)&=&Az(t)+f(t,z(t)),~ t\in[t_0,\infty)_\mathbb{T},\\
		z(t_0)&=&z_0
	\end{eqnarray*}
	both mild and classical solutions are discussed.\\
	Motivated by the above we study the following abstract integro-dynamic equation \begin{eqnarray*}
		y^{\Delta}(s) &=&Ay(s)+\mathcal{F}\left(s,y(s),\int\limits_{t_0}^{s}{\mathcal{H}(s,\tau,y(\tau))}\Delta\tau\right),~ s\in\mathbb{T}\\
		y(0)&=&y_0
	\end{eqnarray*} 
	where $A$ is the infinitesimal generator of a $C_0-$semigroup of bounded linear operators, $T=\{T(s):s\in\mathbb{T}\}\subset\mathcal{L}(\mathbb{Y}).$\\
 Here, 
$\mathbb{T}$ is a time scale and 
$s_0 , S \in\mathbb{T}.$
\begin{eqnarray*}
\mathbb{T}^k=\left\{
\begin{array}{ll}
\mathbb{T}\setminus\left(\rho\left(\sup\mathscr {T}\right),\sup\mathbb{T}\right)~ \mbox{if } \sup\mathbb{T}<\infty\\
\mathbb{T}, ~\mbox{otherwise},
\end{array} \right.
\end{eqnarray*}
where $\mathbb{T}=\left[s_0 , S\right]_{\mathbb{T}}=\left[s_0 , S\right] \cup \mathbb{T}=\left\{s\in \mathbb{T} :s_0 \leq s\leq S \right\}.$

Throughout the article, $\left(Y, \|\cdot\|_{Y} \right)$ is a Banach space, 
    $\mathscr{M}_{Y}$ denotes the collection of all nonempty and bounded subsets of $Y$
	\section{Preliminaries}
	Let us mention some existing results related to time scale calculus, some fixed point theorems, $C_0$-semigroup and their properties, almost automorphic functions and asymptotically almost automorphic functions and their properties collectively, exponential stability of functions in time scale etc.
	\begin{defn}
		\cite[Definition 1.58]{Bohner2} A function  $f:\mathbb{T}\longrightarrow \mathbb{Y}$  is called a rd-continuous, if it is continuous at every right dense point of $\mathbb{T}$, and also ensures the existence of its limits at all left dense points in $\mathbb{T}$. We denote by $C_{rd}\left(\mathbb{T}, \mathbb{Y} \right),$ the collection of all rd-continuous functions $f:\mathbb{T} \longrightarrow \mathbb{Y}.$ 
		
		We also denote by $BC_{rd}(\mathbb{T}, \mathbb{Y}),$ the collection of all rd-continuous and bounded functions $f:\mathbb{T} \longrightarrow \mathbb{Y}.$
	\end{defn}
	\begin{defn}
		A function $f:\mathbb{T}\times \mathbb{Y}\times \mathbb{Y}$ is said to be an \textbf{rd-continuous} function on $\mathbb{T}\times \mathbb{Y}\times \mathbb{Y}$ if $f(s,\cdot,\cdot)$ is continuous on $\mathbb{Y}\times \mathbb{Y}$ for every $s\in \mathbb {S}$ and $f(\cdot,x,y)$ is rd-continuous on $\mathbb{T}$ for every $(x,y)\in \mathbb{Y}\times \mathbb{Y}.$ 
		Moreover, if the continuity of $f(s,\cdot,\cdot)$ is uniform for every $s\in \mathbb{T},$ then the function $f$ is called uniformly rd-continuous.
	\end{defn}
	\begin{defn}\cite[Definition 2.25]{Bohner2}
		A function $p:\mathbb{T}\longrightarrow\mathbb{R}$ is called a \textbf{regressive} function if $\forall~ s\in \mathbb{T}^k,$  the quantity $1+\mu(s)p(s)$ is always a nonzero quantity, where $\mu:\mathbb{T}\longrightarrow\mathbb{R}$ is the graininess function on $\mathbb{T}$, defined as $\mu(s)=\sigma(s)-s.$ 
		We denote by $\mathcal{R}\left(\mathbb{T},\mathbb{R}\right),$ the collection of all regressive functions $p:\mathbb{T}\longrightarrow\mathbb{R}.$
	\end{defn}
	\begin{defn}
		If $p\in\mathcal{R},$ then the generalised exponential function is defined as $$e_p(s,t)=\exp\left(\int\limits_{t}^{s}{\xi_{\mu(\tau)}(p(\tau))}\Delta \tau\right),~\text{for}~t,s\in \mathbb{T}$$ where the cylinder transformation $\xi_h:\mathbb{C}_h\to\mathbb{Z}_h,$ given by $$\xi_h(z)=\frac{1}{h}\log(1+zh),$$
	\end{defn}
where $\log$ is the principal logarithm function. For $h=0,~\xi_o$ is supposed to be the identical transformation.\\
	For more properties of the generalised exponential function relating to regressive function, refer to \cite{Bohner2, Lizama and Mesquita} etc.
	\begin{defn}\label{defn-2.6}
		 A mapping between two normed linear spaces is considered compact if bounded sets are mapped into relatively compact sets.
	\end{defn}
	\begin{thm}\cite[Lemma 4]{Zhu and Wang} (Arzel\`a-Ascoli Theorem) \label{thm-2.3}
		A subset of $C(\mathbb{T},\mathbb{R})$ which is both equicontinuous and bounded is relatively compact. 
	\end{thm}
	\begin{thm}
		A subset of the space of continuous functions on a compact metric space  is relatively compact if and only if it is bounded and equicontinuous.
	\end{thm}
	\begin{defn}
		An operator $T:\mathbb{X}\to \mathbb{Y}$ is said to be completely continuous if it is continuous and sends a bounded set to a relatively compact set. i.e.,  $T$ continuous as well as compact.
	\end{defn}
	\begin{thm}\cite[Theorem 11.2]{Pata}(Krasniselksi\u{\i} fixed point theorem)\label{thm-2.1} Let $\mathbb Y$ be a Banach space and $B\subset \mathbb Y$ be a nonempty, closed and convex subset of $\mathbb{Y}$. Let $F_1, F_2:B\to \mathbb Y$ be such that
		\begin{itemize}
			\item[(i)] $F_1$ is continuous and $F_1(B)$ is relatively compact ($F_1$ is completely continuous)
			\item[(ii)] $F_2$ is a contraction.
			\item[(iii)] $F_1(y_1)+F_2(y_2)\in B,~\forall ~y_1,y_2\in B.$
		\end{itemize}
		Then $\exists ~\Bar{y}$ such that $T_1(\Bar{y})+T_2(\Bar{y})=\Bar{y}.$
	\end{thm}
	\begin{thm} \cite{Lizama and Mesquita,Li&Wang}\label{thm-2.2}
		{If $\alpha>0,$ then $e_{\ominus\alpha}(t,s)\leq 1,$ $t,s\in\mathbb{T},~ t>s.$} 
	\end{thm}
	\begin{lemma}\label{lemma-2.1}
		Let $\alpha>0,$ then for any fixed $s\in\mathbb{T}$ and $s=-\infty,$ one has the following: $\lim\limits_{t\to +\infty}{e_{\ominus\alpha}(t,s)=0}.$
	\end{lemma}
	\begin{lemma}\cite{Wong and Others}\label{lemma-2.2}
		Let $y,f\in C_{rd}(\mathbb{T},\mathbb R^+)$ with $f$ a nondecreasing function and $g,h\in \mathcal{R}^+(\mathbb S \mathbb R)$ with $g\geq 0,h\geq 0$.  If $$y(s)\leq f(s)+\int\limits_{a}^{s}{h(s)\left[y(t)+\int\limits_{a}^{t}{g(\tau)y(\tau)}\Delta \tau\right]} \Delta s ~\text{for}~ s\in \mathbb{T}^k$$ then the following two conditions hold:
	\end{lemma}
	\begin{itemize}
		\item[a.] $y(s)\leq f(s){\left[1+\int\limits_{a}^{s}{h(s)e_{h+g}(s,a)}\Delta s\right]}  ~\text{for}~ s\in \mathbb{T}^k$\\
		\item[b.] $y(s)\leq f(s)e_{h+g}(s,a)$ for $s\in \mathbb{T}^k.$
	\end{itemize}
	In particular if $f(t)=0,$ then $y(t)=0$ for $s\in \mathbb{T}^k$\\
        Now we recall some results concerning semigroups of linear operators on time scales. 
	\begin{defn}\cite{Hamza and Oraby}
		A time scale $\mathbb{T}$ satisfying $a-b\in \mathbb{T},$ for any $a,b\in \mathbb{T}$ with $a>b$ is called a semigroup time scale, usually denoted by $\mathbb{T}\subseteq\mathbb{R}^{\geq0}.$
	\end{defn}
	\begin{defn}\cite{Hamza and Oraby}
		Let $\mathbb{Y}$ be a Banach space and $\mathbb{T}$ is a time scale containing $0.$ We say that $T:\mathbb{T}_0^+\to\mathcal{L}(\mathbb{Y})$ is strongly continuous if $\|T(s)y-y\|\to 0$ as $s\to 0^+$ for each $y\in\mathbb{Y}.$
	\end{defn}
	\begin{defn}\cite{Hamza and Oraby}\label{def-2.9}
		Let $\mathbb{T}$ be a semigroup time scale containing zero and $\mathcal{L}(\mathbb{Y})$ be the space of all bounded linear operators from $\mathbb{Y}$ into itself. A family $T=\big\{ T(t):t\in\mathbb{T}\big\}\subset\mathcal{L}(\mathbb{Y}),$ $T:\mathbb{T}\to\mathcal{L}(\mathbb{Y})$ is a $C_0-$semigroup if it satisfies the following conditions:
		\begin{enumerate}
			\item $T(s+t)=T(s)T(t),$ for all $s,t\in \mathbb{T}$(the semigroup property).
			\item $T_0=T(0)=I,$ where $I$ is the identity operator on $\mathbb{Y}.$
			\item $\lim\limits_{s\to 0^+}{T(s)y=y},$ i.e., $T(\cdot)y:\mathbb{T}\to \mathbb{Y}$ is continuous at $0$ for each  $y\in \mathbb{Y.}$
		\end{enumerate}
	\end{defn}
	In addition if $\lim\limits_{t\to 0 }\|T(t)-I\|=0,$ the $T$ is called uniformly continuous semigroup. Also if we have one more condition  $^*$ $\|T(s)\|_{\mathbb {S}}\leq 1,$ along with the conditions as in Definition \ref{def-2.9}, then we call $T$ to be the contraction semigroup of class ($C_0$).\\
	
	\indent A linear operator $A$ is called the generator of a $C_0-$semigroup $T$ if $$Ay=\lim\limits_{s\to 0^+}{\frac{T(\mu(t))y-T(s)y}{\mu(t)-s}}, y\in D(A),$$
	where the domain of $A,~D(A)$ is the set of all $y\in \mathbb{Y}$ which the above limit exists uniformly in $t.$\\
	\indent The semigroup $T$ is said to be exponentially stable if there exists $K\geq 1$ and $\alpha> 0$ such that $\|T(t-t_0)\|\leq Ke_{\ominus \alpha}(t,t_0),$ for all $t,t_0\in \mathbb{T}, t>t_0.$\\
	For more details on semigroups on time scale refer to \cite{Hamza and Oraby}.
	\begin{defn}\cite{Kere} \label{defn-2.4}
		Let $A$ be a generator of a $C_0-$semigroup $T=\{T(s):s\in\mathbb{T}\}.$ A function $y:\mathbb{T}\to\mathbb{Y}$ is said to be mild solution of the equation $$y^\Delta(t)=Ay(t)+f(s)$$ if it is rd-continuous and satisfied the integral equation $$y(t)=T(s-s_0)y_0+\int\limits_{s_0}^{s}{T(s-\sigma(t))f(t)}\Delta t.$$
	\end{defn}
	Following the similar definition above, we have the following definition.
	\begin{defn}\label{defn-2.15}
		Let $A$ be the infinitesimal generator of a $C_0$-semigroup $\left\{T(s):s\in\mathbb{T}_0^+\right\}.$ Also assume that $\mathcal{F}$ and $\mathcal{H}$ are functions in $C_{rd}\left(\mathbb{T}\times \mathbb{Y}\times \mathbb{Y},\mathbb{Y}\right)$ and $C_{rd}(\mathbb{T}\times\mathbb{T}\times \mathbb{Y}, \mathbb{Y}),$ respectively. Then a function, $y\in BC_{rd}(\mathbb S, \mathbb Y)$ is a mild solution of \eqref{eq-3.1}-\eqref{eq-3.2} if $y$ satisfies the delta integral equation 
		\begin{equation}
			y(s)=T(s-s_0)y_0+\int\limits_{s_0}^{s}{T(s-\sigma(t))\mathcal{F}\left(t,y(t),\int\limits_{s_0}^{t}{\mathcal{H}(t,\tau,y(\tau))}\Delta\tau\right)}\Delta t.
		\end{equation}
	\end{defn}
	Now we recall some properties of almost automorphic functions on time scales. 
	\begin{defn}\cite[Definition 3.1]{Lizama and Mesquita}
		A time scale $\mathbb{T}$ is called invariant under translations if 
		$$\Pi=\{\tau\in \mathbb{R}:s\pm\tau\in \mathbb{T}, \forall s\in \mathbb{T}\}\neq 0.$$
	\end{defn}
	\begin{rem}
		One can easily verify the fact that a symmetric time scale which has semigroup property and contains zero is also invariant under translation.
	\end{rem}
	\begin{lemma}\cite{Guerekata 1}
		Let $\mathbb{T}$ be invariant under translation time scale. Then 
	\end{lemma}
	\begin{itemize}
		\item[i)] $\Pi\subset\mathbb{T}\Longleftrightarrow0\in\mathbb{T}.$
		\item[ii)] $\Pi\cap \mathbb{T}\longleftrightarrow 0\notin \mathbb{T}.$
	\end{itemize}

	In the following, we present preliminary results concerning almost automorphicity and asymptotically almost automorphicity of functions in the time scale perspective. The concept of almost automorphicity is a more general concept of almost periodic function. For more details on such functions refer to \cite{Lizama and Mesquita}. Asymptotically almost automorphic functions are again some generalization of almost automorphic functions, details of which can be found in \cite{Lizama and Mesquita 1}.
	\begin{defn}\cite[Definition 29]{Guerekata 1} Let $\mathbb{Y}$ be a Banach space and $\mathbb{T}$ be a time scale that is invariant under translation. Then an rd-continuous function $f:\mathbb{T}\to\mathbb{Y}$ is called almost automorphic on $\mathbb{T}$ if for every sequence $(s_n)$ on $\Pi,$ there exists a subsequence $(\tau_n)\subset (s_n)$ such that $$\overline{f}(s)=\lim\limits_{n\to\infty}{f(s+\tau_n)}$$
		is well defined for each $s\in\mathbb{T},$ and $$\lim\limits_{n\to\infty\overline{f}(s-\tau_n)}=f(s)$$
		for each $s\in\mathbb{T}.$
	\end{defn}
	The space of all almost automorphic functions $f:\mathbb{T}\to\mathbb{Y}$ is denoted by $AA(\mathbb{T},\mathbb{Y}).$
	
    It is also a well-known result proved in, that if $\mathbb{T}$ is a symmetric time scale which is also invariant under translation, then the graininess function $\mu:\mathbb{T}\to\mathbb{R_+}$ is almost automorphic.
	
	\begin{rem} The space $AA(\mathbb{T},\mathbb{Y})$ equiped with the norm $\sup\limits_{s\in\mathbb{T}}{\|f(s)\|}$ is a Banach space.
	\end{rem}
	\begin{defn}[Definition 3.20, \cite{Lizama and Mesquita}]
		Let $\mathbb{Y}$ be a (real or complex) Banach space and $\mathbb{T}$ be a symmetric time scale which is invariant under translation. Then an rd-continuous function  $f:\mathbb{T}\times \mathbb{Y}\to \mathbb{Y}$ is called almost automorphic in $s\in\mathbb{T}$ uniformly for $x\in K,$ where $K$ is any compact subset of $\mathbb{Y}$, if for every sequence $(s_n)$ on $\Pi,$ there exists a subsequence $(\tau_n)\subset(s_n)$ such that: 
	\end{defn}
	\begin{equation}\label{eq-2.1}
		\overline{f}(s,y)=\lim\limits_{n\to\infty}{f(s+\tau_n,y)}
	\end{equation} 
	is well defined for each $s\in\mathbb{T}, y\in\mathbb{Y}$ and 
	\begin{equation}\label{eq-2.2}
		\lim\limits_{n\to\infty}{\overline{f}(s-\tau_n,y)}=f(s,y)
	\end{equation}
	for each $s\in\mathbb{T}$ and $y\in\mathbb{Y}.$\\
	\indent We denote by $AA(\mathbb{T}\times\mathbb{Y},\mathbb{Y}),$ the space of all almost automorphic functions $f:\mathbb{T}\times\mathbb{Y}\to \mathbb{Y}$ on time scale $\mathbb{T}.$\\
 Following the similar definitions above, we define almost automorphic functions on a single parameter on functions of the type, $f(\cdot,\cdot,\cdot).$
	\begin{defn}
		An rd-continuous function $f:\mathbb{T}\times Y\times Y$ is said to be an almost automorphic function on $s\in \mathbb{T}$ uniformly for all $(x,y)\in \mathbb{Y}\times\mathbb{Y},$ 
		if for every sequence $(s_n)$ on $\Pi,$ there exists a subsequence $(\tau_n)\subset(s_n)$ such that
	\end{defn}
	\begin{equation}\label{eq-2.4}
		\lim\limits_{n\to\infty}{f(s+\tau_n,x,y)}=\Tilde{f}(s,x,y)
	\end{equation} 
	is well defined for each $s\in\mathbb{T},~x,y\in\mathbb{Y}$ and 
	\begin{equation}\label{eq-2.5}
		\lim\limits_{n\to\infty}{\Tilde{f}(s-\tau_n,x,y)}=f(s,x,y)
	\end{equation}
	for each $s\in\mathbb{T}$ and $x,y\in\mathbb{Y}.$\\
	\indent We denote by $AA(\mathbb{T}\times\mathbb{Y}\times \mathbb{Y},\mathbb{Y}),$ the space of all almost automorphic functions $f:\mathbb{T}\times\mathbb{Y}\times\mathbb{Y}\to \mathbb{Y}$ on time scale $\mathbb{T}.$
	
	\begin{defn}\cite{Lizama and Mesquita}
		An rd-continuous function $f:\mathbb{T}^+\times\mathbb{Y}\times{Y}\to\mathbb{Y}$ is said to be asymptotically almost automorphic if it can be uniquely decomposed as $f=g+\phi,$ where $g\in AA(\mathbb{T}^+\times\mathbb{Y}\times{Y},\mathbb{Y})$ and $\phi\in C_{rd}(\mathbb{T}^+\times\mathbb{Y}\times{Y},\mathbb{Y})$ such that 
		$\lim\limits_{s\to\infty}\|\phi(s,x,y)\|=0,$ for all $(x,y)\in\mathbb{Y}\times\mathbb{Y}.$
	\end{defn}
	The set of all functions, $f:\mathbb{T}^+\times\mathbb{Y}\times\mathbb{Y}\to\mathbb{Y}$  which are asymptotically almost automorphic is denoted by $AAA(\mathbb{T}^+\times\mathbb{Y}\times\mathbb{Y},\mathbb{Y}).$\\
	\textbf{Note:} We denote it by $C_{rd_0}(\mathbb{T}^+\times\mathbb{Y}\times\mathbb{Y},\mathbb{Y})$ being the set of all functions, $f\in C_{rd}(\mathbb{T}^+\times\mathbb{Y}\times{Y},\mathbb{Y}) $ such that $\lim\limits_{s\to\infty}\|\phi(s,x,y)\|=0,$ for all $(x,y)\in\mathbb{Y}\times\mathbb{Y}.$
	\begin{rem}
	If $f=g+\phi$ is asymptotically almost automorphic such that $g$ is principal term and $\phi$ is corrective term, then $$\|f\|=\sup\limits_{t\in\mathbb{T}}\|g(t)\|_{\mathbb Y}+\sup\limits_{t\in\mathbb T^+}\|\phi(t)\|_{\mathbb Y}$$ defines a norm such that $(AAA(\mathbb T^+\times\mathbb Y),\|\cdot\|)$ is a Banach space.
	\end{rem}
	\section{Main Results} In Our first approach, we investigate the existence and uniqueness of the mild solution of the following abstract integro-dynamic IVP
	\begin{eqnarray}
		y^{\Delta}(s) &=&Ay(s)+\mathcal{F}\left(s,y(s),\int\limits_{t_0}^{s}{\mathcal{H}(s,\tau,y(\tau))}\Delta\tau\right), \label{eq-3.1}\\
		y(s_0)&=&y_0 \label{eq-3.2}.
	\end{eqnarray}
 where $s\in\mathbb{T}^k.$\\
 \begin{lemma}
     Let $A$ be the infinitesimal generator of the $C_0$-semigroup $\left\{T(s):s\in\mathbb{T}_0^+\right\}.$ Also assume that $\mathcal{F}$ and $\mathcal{H}$ are functions in $C_{rd}\left(\mathbb{T}\times \mathbb{Y}\times \mathbb{Y},\mathbb{Y}\right)$ and $C_{rd}(\mathbb{T}\times\mathbb{T}\times \mathbb{Y}, \mathbb{Y}),$ respectively. Then $y$ is a mild solution of \eqref{eq-3.1}-\eqref{eq-3.2} iff $y$ satisfies the $\Delta-$integral equation 
		\begin{equation}
			y(s)=T(s-s_0)y_0+\int\limits_{s_0}^{s}{T(s-\sigma(t))\mathcal{F}\left(t,y(t),\int\limits_{s_0}^{t}{\mathcal{H}(t,\tau,y(\tau))}\Delta\tau\right)}\Delta t.
		\end{equation}
 \end{lemma}
		
		\begin{proof}
			For proof of the lemma we refer \cite[Lemma 3.1]{Bohner 3}.
		\end{proof}
	
	
	In the remainder of this paper, we will consider $\mathbb{T}$ as a symmetric time scale with the semigroup property and contains zero. We say that $A$ generates an exponentially stable $C_0-$semigroup, i,e. there exist $M\geq 1$ and $\alpha>0$ such that 
	\begin{equation}\label{eq-3.4}
		\|T(s-s_0)\|\leq M e_{\ominus\alpha}(s,s_0), \text{for all}~s,s_0\in\mathbb{T},s\geq s_0.
	\end{equation}  
	
	\begin{thm}\label{thm-3.1}
		Consider the following hypothesis
		\begin{itemize}
			\item[(H$_1$)] Let $\mathcal{F}:\mathbb{T}\times Y\times Y\to Y$ rd-continuous function such that $$\|\mathcal{F}(s,x_1,y_1)-\mathcal{F}(s,x_2,y_2)\|\leq L_{\mathcal{F}}(s)\left(\|x_1-x_2\|+\|y_1-y_2\|\right)$$ for all $t\in\mathbb{T}$ and $x_i,y_i\in Y,~L_{\mathcal{F}}\in\mathcal{R^+}(\mathbb{T},\mathbb{R^+}).$
			\item[(H$_2$)] $\mathcal{H}:\mathbb{T}\times\mathbb{T}\times Y\to Y$ is an rd-continuous function on its first and second variable and continuous on its third variable with $$\|\mathcal{H}(t,s,y_1)-\mathcal{H}(t,s,y_2)\|\leq L_{\mathcal{H}}(s)\|y_1-y_2\|~\forall~ t,s\in\mathbb{T},~y_i\in Y~\text{where}~L_{\mathcal{H}}\in\mathcal{R^+}(\mathbb{T},\mathbb{R^+}).$$
			\item[(H$_3$)] $A$ is the generator of an exponentially stable $C_0$-semigroup, $\left\{T(s):s\in\mathbb{T}_0^+\right\}.$
			\item[(H$_4$)] $M(S-s_0)L_{\mathcal{F}}^*(1+L_{\mathcal{H}}^*(S-s_0))<1.$
		\end{itemize}
		Then \eqref{eq-3.1}-\eqref{eq-3.2} has a unique mild solution whenever $M_\mathcal{F}=\sup\big\{\|\mathcal{F}(s,0,\Psi)\|_{\mathbb Y};~s\in \mathbb{T},~\Psi \in \mathbb Y \big\}<\infty.$ 
	\end{thm}
	\begin{proof} Define a ball, $B_k\subset C_{rd}(\mathbb{T}\times\mathbb{Y})$ as $B_k=\left\{y\in C_{rd}(\mathbb{T}\times\mathbb{Y}):\|y\|_{\mathbb{Y}}\leq k\right\},$ where $k=2M(\|y_0\|+M_{\mathscr{F}}).$
		Let us also define a function $\mathcal{W}:B_k\to  C_{rd}(\mathbb{T}\times\mathbb{Y})$ defined as
		\begin{equation}\label{eq-3.5}
			\mathcal{W}(y)(s):=T(s-s_0)y_0+\int\limits_{s_0}^{s}{T(s-\sigma(t))\mathcal{F}\left(t,y(t),\int\limits_{s_0}^{t}{\mathcal{H}(t,\tau,y(\tau))}\Delta\tau\right)}\Delta t.
		\end{equation}
		In order to apply the Krasnoselski fixed point theorem given by Theorem \ref{thm-2.1}, we express $\mathcal{W}$ as $$\mathcal{W}(y)(s)=\mathcal{W}_1(y)(s)+\mathcal{W}_2(y)(s),$$ where 
		\begin{equation}\label{eq-3.6}
			\mathcal{W}_1(y)(s):=T(s-s_0)y_0+\int\limits_{s_0}^{s}{T(s-\sigma(t))\mathcal{F}\left(t,0,\int\limits_{s_0}^{t}{\mathcal{H}(t,\tau,0)}\Delta\tau\right)}\Delta t.
		\end{equation}
		and 
		\begin{eqnarray}
			\mathcal{W}_2(y)(s):&=\int\limits_{s_0}^{s}T(s-\sigma(t))\left[\mathcal{F}\left(t,y(t),\int\limits_{s_0}^{t}{\mathcal{H}(t,\tau,y(\tau))}\Delta\tau\right)\notag\right.\notag\\
			&~-\left.\mathcal{F}\left(t,0,\int\limits_{s_0}^{t}{\mathcal{H}(t,\tau,0)}\Delta\tau\right)\right]\Delta t.\label{eq-3.7}.
		\end{eqnarray}
		It is obvious to see that $\mathcal{W}_1$ is continuous, We show that $\mathcal{W}_1$ is completely continuous and $\mathcal{W}_2$ is a contraction.\\
		
		\textbf{Step 1:} $\mathcal{W}_1:B_k\to  C_{rd}(\mathbb{J}\times\mathbb{Y})$ is completely continuous.
		For $y\in B_k,$ we have 
		\begin{eqnarray*}
			\|\mathcal{W}_1(y)\|_{\mathbb Y}&=&\Bigg\|T(s-s_0)y_0+\int\limits_{s_0}^{s}{T(s-\sigma(t))\mathcal{F}\left(t,0,\int\limits_{0}^{t}{\mathcal{H}(t,\tau,0)}\Delta\tau\right)}\Delta t\Bigg\|_{\mathbb Y}\\
			&\leq&\|T(s-s_0)\|_{\mathbb Y}~\|y_0\|_{\mathbb Y}+\int\limits_{s_0}^{s}{\Bigg\|T(s-\sigma(t))\mathcal{F}\left(t,0,\int\limits_{s_0}^{t}{\mathcal{H}(t,\tau,0)}\Delta\tau\right)}\Bigg\|_{\mathbb Y}\Delta t\\
			&\leq& M e_{\ominus\alpha}(s,s_0)\|y_0\|_{\mathbb Y}+ M M_{\mathcal{F}}\int\limits_{s_0}^{S}{e_{\ominus\alpha}}(s,\sigma(t))\Delta t\quad (\text{using equation \eqref{eq-3.4}})\\
			&\leq& M(\|y_0\|_{\mathbb {Y}}+M_{\mathcal{F}}(S-s_0))\quad(\text{using Theorem \ref{thm-2.2}}).
		\end{eqnarray*}
		Hence we see that $\mathcal{W}_1$ is bounded in $B_k.$
		
		Next, we test equicontinuity of $\mathcal{W}_1(B_k)$. 
		Let $s_1, s_2\in \mathbb{T}$ and $y\in B_k.$ Then
		\begin{eqnarray*}
			\|\mathcal{W}_1(y)(s_2)-\mathcal{W}_1(y)(s_1)\|_{\mathbb{Y}}&=&\Bigg\|T(s_2-s_0)y_0+\int\limits_{s_0}^{s_2}{T(s_2-\sigma(t))\mathcal{F}\left(t,0,\int\limits_{0}^{t}{\mathcal{H}(t,\tau,0)}\Delta\tau\right)}\Delta t\\
			&-&T(s_1-s_0)y_0+\int\limits_{s_0}^{s_1}{T(s_1-\sigma(t))\mathcal{F}\left(t,0,\int\limits_{s_0}^{t}{\mathcal{H}(t,\tau,0)}\Delta\tau\right)}\Delta t\Bigg\|_{\mathbb{Y}}\\
			&\leq&\|T(s_2-s_0)-T(s_1-s_0)y_0\|_{\mathbb{Y}}\\ 
			&+&\Bigg\|\int\limits_{s_0}^{s_2}{T(s_2-\sigma(t))\mathcal{F}\left(t,0,\int\limits_{s_0}^{t}{\mathcal{H}(t,\tau,0)}\Delta\tau\right)}\Delta t\Bigg\|_{\mathbb{Y}}\\
			&+& \Bigg\|\int\limits_{s_0}^{s_1}{T(s_1-\sigma(t))\mathcal{F}\left(t,0,\int\limits_{0}^{t}{\mathcal{H}(t,\tau,0)}\Delta\tau\right)}\Delta t\Bigg\|_{\mathbb{Y}}\\
			&\leq&\|T(s_2-s_0)-T(s_1-s_0)y_0\|_{\mathbb{Y}}\\ 
			&+& \Bigg\|T(s_2-s_0)\int\limits_{s_0}^{s_1}{T(s_0-\sigma(t))\mathcal{F}\left(t,0,\int\limits_{s_0}^{t}{\mathcal{H}(t,\tau,0)}\Delta\tau\right)}\Delta t\Bigg\|_{\mathbb{Y}}\\
			&+&\Bigg\| T(s_2-s_0)\int\limits_{s_1}^{s_2}{T(s_0-\sigma(t))\mathcal{F}\left(t,0,\int\limits_{s_0}^{t}{\mathcal{H}(t,\tau,0)}\Delta\tau\right)}\Delta t\Bigg\|_{\mathbb{Y}}\\
			&+&\Bigg\|T(s_1-s_0) \int\limits_{0}^{s_1}{T(s_0-\sigma(t))\mathcal{F}\left(t,0,\int\limits_{s_0}^{t}{\mathcal{H}(t,\tau,0)}\Delta\tau\right)}\Delta t\Bigg\|_{\mathbb{Y}}
   \end{eqnarray*}
   \begin{eqnarray*}
            \Rightarrow\|\mathcal{W}_1(y)(s_2)-\mathcal{W}_1(y)(s_1)\|_{\mathbb{Y}}&\leq&\big\|T(s_2-s_0)-T(s_1-s_0)\big\|_{\mathbb{Y}}\left(\|y_0\|_\mathbb{Y}+\Bigg\|\int\limits_{s_0}^{s_1}T(s_0-\sigma(t))\right.\\
            &&\qquad\qquad\qquad\qquad\left.\mathcal{F}\left(t,0,\int\limits_{s_0}^{t}{\mathcal{H}(t,\tau,0)}\Delta\tau\right)\Delta t\Bigg\|_{\mathbb{Y}}\right)\\
			&+& \Bigg\| \int\limits_{s_1}^{s_2}{T(s_2-\sigma(t))\mathcal{F}\left(t,0,\int\limits_{s_0}^{t}{\mathcal{H}(t,\tau,0)}\Delta\tau\right)}\Delta t\Bigg\|_{\mathbb{Y}}\\
			&\leq&\big\|T(s_2-s_0)-T(s_1-s_0)\big\|_{\mathbb{Y}}\left(\|y_0\|_{\mathbb{T}}+MM_{\mathcal{F}}\int\limits_{s_0}^{s_1}{e_{\ominus\alpha}(0,\sigma(t))}\right) \\
			&+&M_{\mathcal{F}}\int\limits_{s_1}^{s^2}{e_{\ominus\alpha}(s_2,\sigma(t))}\Delta t\\
			&\leq&\big\|T(s_2-s_0)-T(s_1-s_0)\big\|_{\mathbb{Y}}\left(\|y_0\|_{\mathbb{T}}+MM_{\mathcal{F}}\int\limits_{s_0}^{s_1}{e_{\alpha}(\sigma(t),0)}\right) \\
			&+&M_{\mathcal{F}}(s_2-s_1) \quad(\text{By using Theorem \ref{thm-2.2}}).
		\end{eqnarray*}
        We will have a similar inequality when we take $t_1>t_2.$ Since $T$ represents a $C_0-$semigroup, $T$ is continuous, and hence the first part of the above inequality tends to zero as $s_2$ tends to $s_1.$ Thus it is obvious that the right-hand side of the above inequality tends to zero as $s_2$ tends to $s_1,$ thus confirming the equicontinuity of $\mathcal{W}_1(B_k)$ by compact mapping theorem. 
		
		Since $\mathcal{W}_1(B_k)$ is both equicontinuous and bounded, by Arzela Ascoli Theorem, ${W}_1$ is compact. Since every compact operator is also completely continuous and subsequently it is completely continuous.\\
		
		\textbf{Step 2:} We show that $\mathcal{W}_2:B_k\to  C_{rd}(\mathbb{T}\times\mathbb{Y})$ is a contraction. 
		Let $x,y\in B_k,$ then we have
		\begin{eqnarray*}
			\|\mathcal{W}[y](s)-\mathcal{W}[x](s)\|_{\mathbb Y}
			&\leq&\int\limits_{s_0}^{s}\|T(s-\sigma(t))\|_{\mathbb {Y}}\left\|\mathcal{F}\left(t,y(t),\int\limits_{s_0}^{t}{\mathcal{H}(t,\tau,y(\tau))}\Delta\tau\right)\right.\\
			&&\left.\mathcal{F}\left(t,x,\int\limits_{s_0}^{t}{\mathcal{H}(t,\tau,x(\tau))}\Delta\tau\right)\right.\|_{\mathbb{Y}}\Delta t.\\
			&\leq& M\int\limits_{s_0}^{s}{e_{\ominus \alpha}(s,\sigma(t))}L_{\mathcal{F}}(t)\left(\|y(t)-x(t)\|+\int\limits_{s_0}^{t}{L_{\mathcal{H}}(\tau)\|y(\tau)-x(\tau)\|_{\mathbb Y}}\Delta \tau\right) \Delta t\\
            &\leq& M\int\limits_{s_0}^{s}e_{\ominus \alpha}(s,\sigma(t))L_{\mathcal{F}}(t)\left(1+L_{\mathcal{H}}^*(t-0)\right) \|y(t)-x(t)\|_{\mathbb Y}\Delta t\\
            &\leq& ML_{\mathcal{F}}^*\left(1+L_{\mathcal{H}}^*(S-s_0)\right)\|y-x\|_{\mathbb Y}\int\limits_{s_0}^{s}e_{\ominus \alpha}(s,\sigma(t)) \Delta t\\
            &\leq& \ ML_{\mathcal{F}}^*(1+L_{\mathcal{H}}^*(S-s_0))\|y-x\|_{\mathbb Y} <\|y-x\|_{\mathbb Y}.
   \end{eqnarray*}
		By the conditions $(H_4),$ we see that $\mathcal{W}_2$ is a contraction.\\
		
		\textbf{Step 3:} For $x,y\in B_k,$ $\mathcal{W}_1(x)(s)+\mathcal{W}_1(y)(s)\in B_k,~ \forall ~s\in\mathbb{T}^r.$ 
		We have \begin{eqnarray*}
			\big\|\mathcal{W}_1[y](s)+\mathcal{W}_2[x](s)\big\|_{\mathbb Y}&=& \Bigg\|T(s-s_0)y_0+\int\limits_{s_0}^{s}{T(s-\sigma(t))\mathcal{F}\left(t,0,\int\limits_{s_0}^{t}{\mathcal{H}(t,\tau,0)}\Delta\tau\right)}\Delta t \\
			&+&\int\limits_{s_0}^{s}{T(s-\sigma(t))\left[\mathcal{F}\left(t,y(t),\int\limits_{s_0}^{t}{\mathcal{H}(t,\tau,y(\tau)}\right)\Delta\tau\right)}\\
			&-&\mathcal{F}\left(t,0,\int\limits_{s_0}^{t}{\mathcal{H}(t,\tau,0)}\Delta\tau\right)\Bigg]\Delta t\Bigg\|_{\mathbb{Y}}\\
			&\leq&\|T(s-s_0)y_0\|+\Bigg\|\int\limits_{s_0}^{s}{T(s-\sigma(t))\mathcal{F}\left(t,0,\int\limits_{s_0}^{t}{\mathcal{H}(t,\tau,0)}\Delta\tau\right)}\Delta t\Bigg\|\\
			&+&\int\limits_{s_0}^{s}{\|T(s-\sigma(t))\|\Bigg\|\mathcal{F}\left(t,y(t),\int\limits_{s_0}^{t}{\mathcal{H}(t,\tau,y(\tau)}\right)\Delta\tau}\\
			\qquad \qquad &-&\mathcal{F}\left(t,0,\int\limits_{s_0}^{t}{\mathcal{H}(t,\tau,0)}\Delta\tau\right)\Bigg\|\Delta t_{\mathbb{Y}}\\
			&\leq& M\|y_0\|+ MM_{\mathcal{F}}+(S-s_0)(L_{\mathcal{F}}^*(1+ L_{\mathcal{H}}^*(S-s_0))\|x\|\\
			&=& M(\|y_0\|+ M_{\mathcal{F}})+M(S-s_0)(L_{\mathcal{F}}^*(1+ L_{\mathcal{H}}^*(S-s_0))k\\
			&\leq& k.
		\end{eqnarray*}
		
		\textbf{Step 4:} If possible let us suppose that $y_1, y_2$ are two distinct solutions of the IVP, then for any $ s\in\mathbb{T},$ using (H$_1$)-(H$_3$), we get 
		\begin{eqnarray*}
			\|y(s)-x(s)\|_{\mathbb Y}
			&\leq&\int\limits_{s_0}^{s}\|T(s-\sigma(t))\|_{\mathbb {Y}}\left\|\mathcal{F}\left(t,y(t),\int\limits_{s_0}^{t}{\mathcal{H}(t,\tau,y(\tau))}\Delta\tau\right)\right.\\
			&&\left.\mathcal{F}\left(t,x,\int\limits_{s_0}^{t}{\mathcal{H}(t,\tau,x(\tau))}\Delta\tau\right)\right.\|_{\mathbb{Y}}\Delta t.\\
			&\leq& M\int\limits_{s_0}^{s}{e_{\ominus \alpha}(s,\sigma(t))}L_{\mathcal{F}}(t)\left(\|y(t)-x(t)\|+\int\limits_{s_0}^{t}{L_{\mathcal{H}}(\tau)\|y(\tau)-x(\tau)\|_{\mathbb Y}}\Delta \tau\right) \Delta t\\
			&\leq& \int\limits_{s_0}^{s}M\frac{ (1+\Tilde{\mu}\alpha)}{\alpha}L_{\mathcal{F}}(t)\left(\|y(t)-x(t)\|+\int\limits_{s_0}^{t}{L_{\mathcal{H}}(\tau)\|y(\tau)-x(\tau)\|_{\mathbb Y}}\Delta \tau\right) \Delta t.
		\end{eqnarray*}
		Using the \ref{lemma-2.2}, from the above inequality we get
		$\|y(s)-x(s)\|_{\mathbb Y}\leq 0,$ which gives $y=x.$ Hence the theorem.
	\end{proof}
 Inspired by the definition of bi-almost automorphic function in $\mathbb{R}$ as in \cite{Xiao Zhu Liang}, we define the following:
	\begin{defn}\textbf{(Bi-almost automorphic function)} A function $H(s,t):\mathbb{T}\times\mathbb{T}\to \mathbb{Y}$ which is rd-continuous with respect to both its variables, is called bi-almost automorphic if for every sequence $(s_n)$ on $\Pi,$ there exists a subsequence $(\tau_n)\subset(s_n)$ such that
		\begin{equation}
			\Tilde{H}(s,t)=\lim\limits_{n\to\infty}{H(s+\tau_n,t+\tau_n)}
		\end{equation} is well defined for each $s,t\in\mathbb{T}$ and \begin{equation}
			\lim\limits_{n\to\infty}{\Tilde{H}(s-\tau_n,t-\tau_n)}=H(s,t)
		\end{equation} for each $s,t\in\mathbb{T}.$
	\end{defn}
	By $bAA(\mathbb{T}\times\mathbb{T},\mathbb{Y}),$ we denote the set of all those bi-almost automorphic functions.
	\begin{rem}
		The notion of bi-almost automorphicity is the generalization of the function $H(s,t)$ having the same period with respect to both of its variables. i.e., $H(s+T,t+T)=H(s,t)~\forall s,t\in \mathbb{T}~\text{for some}~T\in \mathbb{R}-\{0\}.$
	\end{rem} 
	\begin{defn}
		A function $H(s,t,y):\mathbb{T}\times\mathbb{T}\times\mathbb{Y}\to \mathbb{Y}$ which is rd-continuous in its first and second variable, is called bi-almost automorphic if for every sequence $(s_n)$ on $\Pi,$ there exists a subsequence $(\tau_n)\subset(s_n)$ such that
		\begin{equation}
			\Tilde{H}(s,t,y)=\lim\limits_{n\to\infty}{H(s+\tau_n,t+\tau_n,y)}
		\end{equation} is well defined for each $s,t\in\mathbb{T}$ uniformly in $\mathbb{Y}$ and 
		\begin{equation}
			\lim\limits_{n\to\infty}{\Tilde{H}(s-\tau_n,t-\tau_n,y)}=H(s,t,y)
		\end{equation} for each $s,t\in\mathbb{T}$ uniformly in $\mathbb{Y}.$
	\end{defn}
	By $bAA(\mathbb{T}\times\mathbb{T}\times\mathbb{Y},\mathbb{Y}),$ we denote the set of all those bi-almost automorphic functions.
	\begin{defn}[\textbf{bi-asymptotically almost automorphic function}]
		A function $f:\mathbb{T}\times\mathbb{T}\to\mathbb{Y},$ which is rd-continuous with respect to both of its variables, is said to be bi-asymptotically almost automorphic if the function $f(s,t)$ has a unique decomposition, $f(s,t)=g(s,t)+h(s,t)$ with $g\in bAA(\mathbb{T}\times\mathbb{T},\mathbb{Y})$ and $h\in C_{rd_0}(\mathbb{T}^+\times\mathbb{T}^+),$ i.e., $h$ is rd-continuus with respect to both the variables and $\lim\limits_{(s,t)\to (\infty,\infty)}{h(s,t)}=0.$
	\end{defn}
	In the following we establish results concerning the existence and uniqueness of the bounded, asymptotically almost automorphic solution to the given IVP; \eqref{eq-3.1}-\eqref{eq-3.2}.
	Let $y\in BC_{rd}(\mathbb{T},\mathbb{Y})$  and consider the following hypothesis
	\begin{itemize}
		\item[(H$_1^{'}$)] Let $\mathcal{F}(s,x,y)=\mathcal{G}(s,x,y)+\mathcal{I}(s,x,y)\in AAA(\mathbb{S^+}\times Y\times Y,Y)$ be a function satisfying,  $$\|\mathcal{F}(s,x_1,y_1)-\mathcal{F}(s,x_2,y_2)\|\leq L_{\mathcal{F}}(s)\left(\|x_1-x_2\|+\|y_1-y_2\|\right)$$
		for all $s\in\mathbb{T}$ and $x_i,y_i\in Y,~L_{\mathcal{F}}\in AA(\mathbb{T},\mathbb{R^+}).$
		\item[(H$_2^{'}$)] $\mathcal{H}(s,\tau,y(\tau))=\mathcal{J}(s,\tau,y(\tau))+\mathcal{K}(s,\tau,y(\tau))\in bAAA(\mathbb{T}^+\times\mathbb{T}^+\times \mathbb{Y}, \mathbb{Y})$ be such that\\ $\lim\limits_{s\to\infty}\int\limits_{-\infty}^{s}{\mathcal{K}(s,\tau,y(\tau))}\Delta \tau=0$ and $$\|\mathcal{H}(t,s,y_1)-\mathcal{H}(t,s,y_2)\|\leq L_{\mathcal{H}}(s)\|y_1-y_2\|~\forall~ t,s\in\mathbb{T},~y_i\in Y,$$ $\text{where}~L_{\mathcal{H}}\in AA(\mathbb{T},\mathbb{R^+})$ such that $L_{\mathcal{H}}^{1}(t)=\int\limits_{-\infty}^{s}{L_{\mathcal{H}}(t)}\Delta t<\infty,$ for any $s\in\mathbb{T}.$
		\item[(H$_3^{'}$)] $A$ is the generator of an exponentially stable $C_0$-semigroup, $\left\{T(s):s\in\mathbb{T}_0^+\right\}.$
		\item[(H$_4^{'}$)] There exists $r>0$ such that $\frac{\alpha r}{M}-r(1+\Tilde{\mu}\alpha)\left(L_{\mathcal{F}}^*+{L_{\mathcal{H}}^{1}}^*\right)>(1+\Tilde{\mu}\alpha)M_F,$ where\\ $M_\mathcal{F}=\sup\big\{\|\mathcal{F}(s,0,z)\|_{\mathbb Y};~s\in \mathbb{T}^+,~z \in \mathbb Y \big\},\text{where,}~L_{\mathcal{F}}^*=\sup\limits_{t\in \mathbb{T}^+}{L_{\mathcal{F}}(t)},~ {L_{\mathcal{H}}^1}^*=\sup\limits_{t\in \mathbb{T}^+}{L_{\mathcal{H}^{1}}(s)}.$ 
	\end{itemize}
	We prove the existence and uniqueness of the solution of the IVP \eqref{eq-3.1}-\eqref{eq-3.2}. Let us first establish some results required for the main result.
	\begin{lemma}\label{lemma-3.2}
		Let $A$ be the infinitesimal generator of the $C_0$-semigroup $T=\left\{T(s):s\in\mathbb{T}_0^+\right\}.$ Also assume that $\mathcal{F}$ and $\mathcal{H}$ are functions in $C_{rd}\left(\mathbb{T}\times \mathbb{Y}\times \mathbb{Y},\mathbb{Y}\right)$ and $C_{rd}(\mathbb{T}\times\mathbb{T}\times \mathbb{Y}, \mathbb{Y}),$ respectively. Then $y\in BC_{rd}(\mathbb {S},\mathbb{Y})$ is a mild solution of \eqref{eq-3.1}-\eqref{eq-3.2} if and only if $y$ satisfies the following improper $\Delta-$integral
		\begin{equation} \label{eq-3.8}
			y(s)=\int\limits_{-\infty}^{s}{T(s-\sigma(t)}\mathcal{F}\left(t,y(t),\int\limits_{-\infty}^{t}{\mathcal{H}(t,\tau,y(\tau))\Delta \tau}\right)\Delta t.
		\end{equation}
	\end{lemma}
	\begin{proof}
		If $y$ is mild solution of \eqref{eq-3.1}-\eqref{eq-3.2} then by Definition \ref{defn-2.15} we have 
		\begin{equation}\label{eq-3.9}
			y(s)=T(s-s_o)y_0+\int\limits_{s_0}^{s}{T(s-\sigma(t))\mathcal{F}\left(t,y(t),\int\limits_{s_0}^{t}{\mathcal{H}(t,\tau,y(\tau))}\Delta\tau\right)}\Delta t.
		\end{equation}
		Since $T$ is exponentially stable, so  we get 
		\begin{equation}\label{eq-3.10}
			\|T(s-s_0)y_0\|=Me_{\ominus \alpha}(s,s_0)\|y_0\|.
		\end{equation}
		Again, since $y_0=y(s_0)$ and $y\in BC_{rd}(\mathbb{T},\mathbb{Y}),$ there exists $m>0$ such that $\|y\|_{\mathbb{Y}}\leq m$ and hence from \eqref{eq-3.10}, we get
		\begin{equation}
			\|T(s-s_0)y_0\|=Mme_{\ominus \alpha}(s,s_0),~s\geq 0.
		\end{equation}
		Taking $\lim s_0\to-\infty,$ we can see from \eqref{eq-3.10} that
		\begin{equation}
			\lim\limits_{s_0\to-\infty}{\|T(s-s_0)y_0\|}=0.
		\end{equation}
		Now taking $\lim {s_0 \to-\infty}$ in equation \eqref{eq-3.9}, we obtain 
		\begin{equation}\label{eq-3.13}
			y(s)= \int\limits_{-\infty}^{s}{T(s-\sigma(t))\mathcal{F}\left(t,y(t),\int\limits_{-\infty}^{t}{\mathcal{H}(t,\tau,y(\tau))}\Delta\tau\right)}\Delta t.
		\end{equation}
		Now we check for convergence of 
		\begin{equation}\label{eq-3.14}
			\int\limits_{-\infty}^{s}{T(s-\sigma(t))\mathcal{F}\left(t,y(t),\int\limits_{-\infty}^{t}{\mathcal{H}(t,\tau,y(\tau))}\Delta\tau\right)}\Delta t.
		\end{equation}
		Let us consider the following 
		\begin{eqnarray*}
			F_1&=&\mathcal{F}\left(t,0,\int\limits_{s_0}^{t}{\mathcal{H}(t,\tau,0)}\Delta\tau\right)\\
			F_2&=& \mathcal{F}\left(t,y(t),\int\limits_{s_0}^{t}{\mathcal{H}(t,\tau,y(\tau))}\Delta\tau\right)\notag\\
			&\qquad-&\mathcal{F}\left(t,0,\int\limits_{s_0}^{t}{\mathcal{H}(t,\tau,0)}\Delta\tau\right)
		\end{eqnarray*} 
		such that $\mathcal{F}=F_1+F_2.$\\
		
		Now \begin{eqnarray}
			&&\Bigg\|\int\limits_{-\infty}^{s}{T(s-\sigma(t))\mathcal{F}\left(t,y(t),\int\limits_{-\infty}^{t}{\mathcal{H}(t,\tau,y(\tau))}\Delta\tau\right)}\Delta t\Bigg\|\notag\\
                &=&\int\limits_{-\infty}^{s}{\|T(s-\sigma(t))(F_1+F_2)\|}\Delta t\notag\\
			&\leq& M\int\limits_{-\infty}^{s}{e_{\ominus \alpha}(s,\sigma(t))\|F_1+F_2\|}\Delta t\notag\\
			&\leq& M\left\{\int\limits_{-\infty}^{s}{e_{\ominus \alpha}(s,\sigma(t))\|F_1\|}\Delta t+\int\limits_{-\infty}^{s}{e_{\ominus \alpha}(s,\sigma(t))\|F_2\|}\Delta t\right\}\notag\\
			&\leq& M\left\{M_{\mathcal{F}}\int\limits_{-\infty}^{s}\frac{1+\mu(t)\alpha}{\alpha}\left(-(\ominus
			\alpha)e_{\ominus}(s,\sigma(t))\right)\Delta t\right\}\label{eq-3.15}.
		\end{eqnarray}
		We have \begin{eqnarray}
			\int\limits_{-\infty}^{s}{e_{\ominus \alpha}(s,\sigma(t))\|F_1\|}\Delta t &=&\int\limits_{-\infty}^{s}\frac{1+\mu(t)\alpha}{\alpha}\left(-(\ominus
			\alpha)e_{\ominus}(s,\sigma(t))\right)\|F_1\|\Delta t \notag\\
			&\leq&\frac{ M_{\mathcal{F}}(1+\Tilde{\mu}\alpha)}{\alpha}\int\limits_{-\infty}^{s}\left(-(\ominus\alpha)e_{\ominus}(s,\sigma(t))\right)\Delta t \notag\\
			&\leq&\frac{ M_{\mathcal{F}}(1+\Tilde{\mu}\alpha)}{\alpha}\left(e_{\ominus}(s,s)-e_{\ominus}(s,-\infty)\right)\Delta t \notag\\
			&=&\frac{ M_{\mathcal{F}}(1+\Tilde{\mu}\alpha)}{\alpha}\label{eq-3.16}.
		\end{eqnarray}
		Also 
		\begin{eqnarray}
			\int\limits_{-\infty}^{s}{e_{\ominus \alpha}(s,\sigma(t))\|F_2\|}\Delta t &=&  \int\limits_{-\infty}^{s} {e_{\ominus}(s,\sigma(t))}\Bigg\| \mathcal{F}\left(t,y(t),\int\limits_{0}^{t}{\mathcal{H}(t,\tau,y(\tau))}\Delta\tau\right)\notag\\
			&\qquad-&\mathcal{F}\left(t,0,\int\limits_{0}^{t}{\mathcal{H}(t,\tau,0)}\Delta\tau\right)\Bigg\|\Delta t\notag\\
			&\leq& \int\limits_{-\infty}^{s} {e_{\ominus}(s,\sigma(t))}\left(L_{\mathcal{F}}\left(\|y(t)\|+ \int\limits_{-\infty}^{s}{L_{\mathcal{H}}(\tau)\|y(\tau)\|}\right)\right)\notag\\
			&\leq& mL_{\mathcal{F}}\left(1+L_{\mathcal{H}}^{'} \right)\int\limits_{-\infty}^{s} {e_{\ominus}(s,\sigma(t))}\notag\\
			&=& mL_{\mathcal{F}}\left(1+L_{\mathcal{H}}^{'}
			\right)\frac{(1+\Tilde{\mu}\alpha)}{\alpha}\label{eq-3.17}.
		\end{eqnarray}
		Using results given by equation \eqref{eq-3.16} and equation \eqref{eq-3.17}, from \eqref{eq-3.15}, we get 
		\begin{eqnarray*}
			\Bigg\|\int\limits_{-\infty}^{s}{T(s-\sigma(t))\mathcal{F}\left(t,y(t),\int\limits_{-\infty}^{t}{\mathcal{H}(t,\tau,y(\tau))}\Delta\tau\right)}\Delta t\Bigg\|
			&\leq& M\frac{(1+\Tilde{\mu}\alpha)}{\alpha}(M_{\mathcal{F}}+ mL_{\mathcal{F}}(1+L_{\mathcal{H}}^{'})
		\end{eqnarray*} 
		which shows that $\Bigg\|\int\limits_{-\infty}^{s}{T(s-\sigma(t))\mathcal{F}\left(t,y(t),\int\limits_{-\infty}^{t}{\mathcal{H}(t,\tau,y(\tau))}\Delta\tau\right)}\Delta t\Bigg\|$ is convergent.\\
		Now $$y_0=y(s_0)=\Bigg\|\int\limits_{-\infty}^{s_0}{T(s_0-\sigma(t))\mathcal{F}\left(t,y(t),\int\limits_{-\infty}^{t}{\mathcal{H}(t,\tau,y(\tau))}\Delta\tau\right)}\Delta t$$
		and \begin{eqnarray*}
			y(s)&=&\int\limits_{-\infty}^{s}{T(s-\sigma(t))\mathcal{F}\left(t,y(t),\int\limits_{-\infty}^{t}{\mathcal{H}(t,\tau,y(\tau))}\Delta\tau\right)}\Delta t\\
			&=&\int\limits_{-\infty}^{0}{T(s-\sigma(t))\mathcal{F}\left(t,y(t),\int\limits_{-\infty}^{t}{\mathcal{H}(t,\tau,y(\tau))}\Delta\tau\right)}\Delta t\\
			&+& \int\limits_{0}^{s}{T(s-\sigma(t))\mathcal{F}\left(t,y(t),\int\limits_{-\infty}^{t}{\mathcal{H}(t,\tau,y(\tau))}\Delta\tau\right)}\Delta t\\
			&=&T(s-s_0)\int\limits_{-\infty}^{s_0}{T(s_0-\sigma(t))\mathcal{F}\left(t,y(t),\int\limits_{-\infty}^{t}{\mathcal{H}(t,\tau,y(\tau))}\Delta\tau\right)}\Delta t\\
			&+&\int\limits_{s_0}^{s}{T(s-\sigma(t))\mathcal{F}\left(t,y(t),\int\limits_{-\infty}^{t}{\mathcal{H}(t,\tau,y(\tau))}\Delta\tau\right)}\Delta t\\
			&=&T(s-s_0)y(s_0)+\int\limits_{0}^{s}{T(s-\sigma(t))\mathcal{F}\left(t,y(t),\int\limits_{-\infty}^{t}{\mathcal{H}(t,\tau,y(\tau))}\Delta\tau\right)}\Delta t.
		\end{eqnarray*}
		By the above discussion, we can confirm that $y$ given by \eqref{eq-3.8} is in fact a mild solution to the initial value problem given by \eqref{eq-3.1}-\eqref{eq-3.2}.
	\end{proof}
	\begin{prop}\label{prop-1}
		If $f\in AA(\mathbb{T},\mathbb{Y}),$ then the range set $R_f=\left\{f(s):s\in\mathbb{T}\right\}$ is relatively compact.
	\end{prop}
	\begin{prop}\label{prop-2}
		If $f\in AAA(\mathbb{T}^+,\mathbb{Y}),$ then the range set $R_f=\left\{f(s):s\in\mathbb{T}^+\right\}$ is relatively compact.
	\end{prop}
	\begin{prop}\label{prop-4}
		Let $F\in AA(\mathbb{T}\times\mathbb{Y}\times\mathbb{Y},\mathbb{Y})$ be such $$\|F(s,x_1,y_1)-F(s,x_2,y_2)\|\leq L_F(s)(\|x_1-x_2\|+\|y_1-y_2\|$$
		uniformly for $s\in\mathbb{T}~\text{and}~x_i,y_i\in \mathbb{Y},~i=1,2,$ where, $L_F\in AA(\mathbb{T},\mathbb{R^+}).$ 
		Then for any $x,y\in AA(\mathbb{T},\mathbb{Y}),$ the function $\Psi:\mathbb{T}\to\mathbb{Y},$ given by $\Psi(s)=F(\cdot,x,y)$ is almost automorphic.
		\begin{proof}
			Let $(\tau_n^{'} )_{n\in\mathbb{N}}$ be a sequence in $\Pi.$ since $x,y$ and $F$ are almost automorphic functions, we can get a subsequence $(\tau_n)_{n\in\mathbb{N}}$ of $(\tau_n^{'} )_{n\in\mathbb{N}}$ such that
			\begin{enumerate}
				\item $\lim\limits_{n\to\infty}{x(s+\tau_n)}=\Tilde{x}(s)$ exists for each $s\in\mathbb{T}.$\\
				\item $\lim\limits_{n\to\infty}{\Tilde{x}(s-\tau_n)}=x(s)$ exists for each $s\in\mathbb{T}.$\\
				\item $\lim\limits_{n\to\infty}{y(s+\tau_n)}=\Tilde{y}(s)$ exists for each $s\in\mathbb{T}.$\\
				\item $\lim\limits_{n\to\infty}{\Tilde{y}(s-\tau_n)}=y(s)$ exists for each $s\in\mathbb{T}.$\\
				\item $\lim\limits_{n\to\infty}{F(s+\tau_n,x,y)}=\Tilde{F}(s,x,y)$ exists for each $s\in\mathbb{T}.$\\
				\item $\lim\limits_{n\to\infty}{\Tilde{F}(s-\tau_n,x,y)}=F(s,x,y)$ exists for each $s\in\mathbb{T}.$
			\end{enumerate}
			Also since $L_F\in AA(\mathbb{T},\mathbb{R^+}),$ we have
			\begin{itemize}
				\item[*] $\lim\limits_{n\to\infty}{L_F(s+\tau_n)}=\Tilde{L_F}(s)$ exists for each $s\in\mathbb{T}$ 
				\item[**]   $\lim\limits_{n\to\infty}{\Tilde{L_F}(s-\tau_n)}=L_F(s)$ exists for each $s\in\mathbb{T}.$
			\end{itemize}
			Let  $\Tilde{\Psi}(s)=\Tilde{F}(\cdot,\Tilde{x}(s),\Tilde{y}(s)).$
			We have \begin{eqnarray}
				\|\Psi(s+\tau_n)-\Tilde{\Psi}(s)\|&=&\|F(s+\tau_n,x(s+\tau_n),y(s+\tau_n))-\Tilde{F}(\cdot,\Tilde{x}(s),\Tilde{y}(s))\|\notag\\
				&\leq&\|F(s+\tau_n,x(s+\tau_n),y(s+\tau_n))-{F}(s+\tau_n,\Tilde{x}(s),\Tilde{y}(s))\|\notag\\
				&+&\|{F}(s+\tau_n,\Tilde{x}(s),\Tilde{y}(s))-\Tilde{F}(\cdot,\Tilde{x}(s),\Tilde{y}(s))\|\label{eq-3.18}.
			\end{eqnarray}
			According to the given assumptions, we have 
			\begin{align*}
				&\|F(s+\tau_n,x(s+\tau_n),y(s+\tau_n))-{F}(s+\tau_n,\Tilde{x}(s),\Tilde{y}(s))\|\notag\\
				& \leq L_F(s+\tau_n)(\|x(s+\tau_n)-\Tilde{x}(s)\|+\|y(s+\tau_n)-\Tilde{x}(s)\|).
			\end{align*}
			So by $(1),(3)$ and $(*),$ we get
			\begin{equation}\label{eq-3.19}
				\lim\limits_{n\to\infty}\|F(s+\tau_n,x(s+\tau_n),y(s+\tau_n))-\Tilde{F}(s+\tau_n,\Tilde{x}(s),\Tilde{y}(s))\|=0.
			\end{equation}
			Also by $(5),$ we have 
			\begin{equation}\label{eq-3.20}
				\lim\limits_{n\to\infty}\|{F}(s+\tau_n,\Tilde{x}(s),\Tilde{y}(s))-\Tilde{F}(\cdot,\Tilde{x}(s),\Tilde{y}(s))\|=0.
			\end{equation}
			So by using equations \eqref{eq-3.19} and \eqref{eq-3.20}, we get from equation \eqref{eq-3.18}
			\begin{equation*}
				\lim\limits_{n\to\infty}{\Psi(s+\tau_n)}=\Tilde{\Psi}(s)~\text{for each}~s\in\mathbb{T}.
			\end{equation*}
			Using a similar argument as above we can also prove that 
			\begin{equation*}
				\lim\limits_{n\to\infty}{\Tilde{\Psi}(s-\tau_n)={\Psi}(s)}~\text{for each}~s\in\mathbb{T}.
			\end{equation*}
			This proves that $\Psi\in AA(\mathbb{T},\mathbb{Y}).$
		\end{proof}
	\end{prop}
	\begin{prop}\label{prop-5}
		If $\mathcal{F}\in AAA(\mathbb{T}^+\times\mathbb{Y}\times\mathbb{Y},\mathbb{Y})$ and satisfies $(H_2^{'})$, then for $x,y\in AAA(\mathbb{T}^+,\mathbb{Y}),$ the function $\Gamma:\mathbb{T}^+\to\mathbb{Y},$ given by $\Gamma(s)=\mathcal{F}(s,x(s),y(s))$ is also asymptotically almost automorphic.
	\end{prop}
	\begin{proof}
		By $(H_2^{'})$ $\mathcal{F}\in AAA(\mathbb{T}^+\times\mathbb{Y}\times\mathbb{Y},\mathbb{Y}),$ and  $$\mathcal{F}(s,x(s),y(s))=\mathcal{G}(s,x(s),y(s)+\mathcal{I}(s,x(s),y(s)),$$ where $\mathcal{G}\in AA(\mathbb{T}\times\mathbb{Y}\times\mathbb{Y},\mathbb{Y})$ and $\mathcal{I}\in C_{rd_0}(\mathbb{T}\times\mathbb{Y}\times\mathbb{Y},\mathbb{Y}).$ Again for $x,y\in AAA(\mathbb{T}^+,\mathbb{Y}),$ we have $x(s)=u(s)+w(s)$ and $y(s)=v(s)+z(s),$ where $u,v\in AA(\mathbb{T},\mathbb{Y})$ and $w,z\in C_{rd_0}(\mathbb{T},\mathbb{Y}).$\\
		Now, \begin{eqnarray*}
			\mathcal{F}(s,x(s),y(s))&=&\mathcal{F}(s,u(s),v(s))+[\mathcal{F}(s,x(s),y(s))-\mathcal{F}(s,u(s),v(s))].
		\end{eqnarray*}
		Let $F^{'}(s)=\mathcal{F}(s,x(s),y(s))-\mathcal{F}(s,u(s),v(s)),$ $F^{''}(s)=\mathcal{F}(s,u(s),v(s)).$ \\
		Clearly by \textbf{Proposition \ref{prop-4}}, for $u,v\in AA(\mathbb{T},\mathbb{Y})$ $\mathbb{F}^{''}\in AA(\mathbb{T},\mathbb{Y}).$ 
		
		Again, for $w(s)\in C_{rd_0}(\mathbb{T}^+,\mathbb{Y}),$  
		we have $\lim\limits_{s\to\infty}\|x(s)-u(s)\|=\lim\limits_{s\to\infty}\|w(s)\|=0.$\\
		Similarly we get   $\lim\limits_{s\to\infty}\|y(s)-v(s)\|=\lim\limits_{s\to\infty}\|z(s)\|=0.$\\
		Now using $(H_2^{'}),$ we get
		\begin{eqnarray*}
			\|F^{'}(s)\| &=&\|\mathcal{F}(s,x(s),y(s))-\mathcal{F}(s,u(s),v(s))\|\leq L_\mathcal{F}(s)(\|x(s)-u(s)\|+\|y(s)-v(s)\|)\\
			\Rightarrow \lim\limits_{s\to\infty}\|\mathcal{F}^{'}(s)\|&=&\lim\limits_{s\to\infty}L_F(s)(\|x(s)-u(s)\|+\|y(s)-v(s)\|)\\
			&=&\lim\limits_{s\to\infty}(L_F(s))(\lim\limits_{s\to\infty}\|x(s)-u(s)\|+\lim\limits_{s\to\infty}\|y(s)-v(s)\|)\\
			&=&0~(\text{as $L_\mathcal{F}\in AA(\mathbb{T},\mathbb{R}^+)\subset BC_{rd}(\mathbb{T},\mathbb{R}^+)$}).
		\end{eqnarray*}
		Hence $F^{'}\in AA(\mathbb{T},\mathbb{Y}).$
		So from the above discussion, we can confirm that $\Gamma\in AAA(\mathbb{T}^+,\mathbb{Y}).$
	\end{proof}
	
	\begin{prop}\label{Prop-3.6}
		If $\mathcal{J}\in bAA(\mathbb{T}\times\mathbb{T}\times\mathbb{Y},\mathbb{Y})$ be a function as mentioned in  $(H_2^{'})$ then for any $y\in AA(\mathbb{T},\mathbb{Y})$ the function $$\Phi(s):=\int\limits_{-\infty}^{s}{\mathcal{J}(s,\tau,y(\tau))}\Delta \tau$$ is almost automorphic.
	\end{prop}
	\begin{proof}
		Let $(\tau_n)_{n\in \mathbb{N}}$ be a sequence in $\Pi.$ Since $y\in AA(\mathbb{T},\mathbb{Y})$ and $\mathcal{J}\in bAA(\mathbb{T}\times\mathbb{T}\times\mathbb{Y},\mathbb{Y}),$ we have a subsequence $ (s_n)\subset(\tau_n)_{n\in \mathbb{N}}$ such that 
		\begin{enumerate}
			\item[1)] $\lim\limits_{n\to\infty}{y(s+\tau_n)}=\Tilde{y}(s)$ exists for each $s\in\mathbb{T}.$\\
			\item[2)]  $\lim\limits_{n\to\infty}{\Tilde{y}(s-\tau_n)}=y(s)$ exists for each $s\in\mathbb{T}.$\\
			\item[3)]  $\lim\limits_{n\to\infty}{\mathcal{J}(s+\tau_n,t+\tau_n,y)}=\Tilde{\mathcal{J}}(s,t,y)$ exists for each $s,t\in\mathbb{T}.$\\
			\item[4)]  $\lim\limits_{n\to\infty}{\Tilde{\mathcal{J}}(s-\tau_n,t-\tau_n,y)}=\mathcal{J}(s,t,y)$ exists for each $s,t\in\mathbb{T}.$
		\end{enumerate}
		Let $\Tilde{\Phi}(s)=\int\limits_{-\infty}^{s}{\Tilde{\mathcal{J}}(s,\tau,\tilde{y}(\tau))}\Delta \tau.$\\
		Then
		\begin{eqnarray*}
			\|\Phi(s+s_n)-\Tilde{\Phi}(s)\|&=& \left\|\int\limits_{-\infty}^{s+s_n}{\mathcal{J}(s+s_n,\tau,y(\tau))}\Delta \tau-\int\limits_{-\infty}^{s}{\Tilde{\mathcal{J}}(s,\tau,\tilde{y}(\tau))}\Delta \tau\right\|\\
			&=&\left\|\int\limits_{-\infty}^{s}{\mathcal{J}(s+s_n,\tau+s_n,y(\tau+s_n))}\Delta \tau-\int\limits_{-\infty}^{s}{\Tilde{\mathcal{J}}(s,\tau,\tilde{y}(\tau))}\Delta \tau\right\|\\
			&\leq& \left\|\int\limits_{-\infty}^{s}{\mathcal{J}(s+s_n,\tau+s_n,y(\tau+s_n))}\Delta \tau- \int\limits_{-\infty}^{s}{\mathcal{J}(s+s_n,\tau+s_n,\Tilde{y}(\tau))}\Delta \tau\right\|\\
			&+&\left\|\int\limits_{-\infty}^{s}{\mathcal{J}(s+s_n,\tau+s_n,\Tilde{y}(\tau))}\Delta \tau-\int\limits_{-\infty}^{s}{\Tilde{\mathcal{J}}(s,\tau,\tilde{y}(\tau))}\Delta \tau\right\|\\
			&\leq& \int\limits_{-\infty}^{s}\left\|{\mathcal{J}(s+s_n,\tau+s_n,{y}(\tau+s_n))}-\mathcal{J}(s+s_n,\tau+s_n,\Tilde{y}(\tau))\right\| \Delta \tau\\
			&+& \int\limits_{-\infty}^{s}\left\|{J(s+s_n,\tau+s_n,\Tilde{y}(\tau))}-\Tilde{\mathcal{J}}(s,\tau,\Tilde{y}(\tau))\right\| \Delta \tau\\
			&\leq& \int\limits_{-\infty}^{s}\mathcal{\mathcal{J}}(s+s_n)\left\|{y}(s+s_n)-\tilde{y}(s)\right\| \Delta \tau\\
			&+& \int\limits_{-\infty}^{s}\left\|{\mathcal{J}(s+s_n,\tau+s_n,\Tilde{y}(\tau))}-\Tilde{\mathcal{J}}(s,\tau,\Tilde{y}(\tau))\right\| \Delta \tau. 
		\end{eqnarray*}
		Now taking $n\longrightarrow \infty,$ taking into account the fact that $L_{\mathcal{J}}\in AA(\mathbb{T},\mathbb{Y})$ as given by $(H_2);$ together with 1) and 3) we see from the above inequality that, $$ \lim\limits_{n\to \infty}{\Phi(s+s_n)}=\Tilde{\Phi}(s)~ \forall s\in \mathbb{T}.$$
		Similarly by using $H_2$ together with 2) and 4) we can show that $$\lim\limits_{n\to \infty}{\Tilde{\Phi}(s-s_n)}=\Phi(s)~ \forall s\in \mathbb{T}.$$ Thus we establish $\Phi\in AA(\mathbb{T},\mathbb{Y}).$
	\end{proof}
	\begin{prop}\label{prop-7} 
		Let $\mathcal{H}\in bAAA(\mathbb{T}^+\times\mathbb{T}^+\times\mathbb{Y},\mathbb{Y})$  satisfying $(H_2^{'}).$ Then for any $y\in AAA(\mathbb{T}^+,\mathbb{Y})$ the function $$\Phi(s):=\int\limits_{-\infty}^{s}{\mathcal{H}(s,\tau,y(\tau))}\Delta \tau,$$ is asymptotically almost automorphic.
	\end{prop} 
	\begin{proof} Since  $y\in AAA(\mathbb{T},\mathbb{Y}).$ Let us suppose that $y(\tau)=z(\tau)+w(\tau),$ where $z\in AA(\mathbb{T},\mathbb{Y})$ and $w\in C_{rd_0}(\mathbb{T}^+,\mathbb{Y}).$ Again for  $\mathcal{H}\in bAAA(\mathbb{T}\times\mathbb{T}\times\mathbb{Y},\mathbb{Y}),$ we have 
		\begin{equation}\label{eq-3.25}
			\mathcal{H}(s,\tau,y(\tau))=\mathcal{J}(s,\tau,y(\tau))+\mathcal{K}(s,\tau,y(\tau))
		\end{equation}
		for some(unique) $\mathcal{J}\in bAA(\mathbb{T}\times\mathbb{T}\times\mathbb{Y},\mathbb{Y})$ and $\mathcal{K}\in C_{rd_0}(\mathbb{T}^+\times\mathbb{T}^+\times\mathbb{Y},\mathbb{Y}).$
		Now, we have 
		\begin{equation}\label{eq-3.26}
			\mathcal{H}(s,\tau,y(\tau))=\mathcal{H}(s,\tau,z(\tau))+[\mathcal{H}(s,\tau,y(\tau))-\mathcal{H}(s,\tau,z(\tau))].
		\end{equation}
		It is evident by using similar arguments as in \textbf{Proposition \ref{prop-4}} that  $\Phi_1(s,\tau):=\mathcal{H}(s,\tau,z(\tau))\in bAA(\mathbb{T}\times\mathbb{T},\mathbb{Y})$ for $z$ being almost automorphic.
		Also by using the condition $(H_2^{'})$ and the fact that $w(s)=x(s)-z(s),$ we can also confirm that
		\begin{eqnarray*} 
			\lim\limits_{(s,\tau)\to(\infty,\infty)}\|\mathcal{H}(s,\tau,y(\tau))-\mathcal{H}(s,\tau,z(\tau))\|&=&\lim\limits_{\tau\to\infty}|L_{\mathcal{H}}(\tau)|\|y(\tau)-z(\tau)\|\\
			&=&\lim\limits_{\tau\to\infty}|L_{\mathcal{H}}(\tau)|\|w(\tau)\|=0.
		\end{eqnarray*}
		Thus we have from above discussion, $\Phi_2(s,\tau):=\mathcal{H}(s,\tau\,y(\tau))\in bAAA(\mathbb{T}^+\times\mathbb{T}^+,\mathbb{Y}).$ Now since $\Phi_2\in  bAAA(\mathbb{T}^+\times\mathbb{T}^+,\mathbb{Y}),$ we have unique decomposition of $\Phi_2.$ \\
		Let $\Phi_2(s,\tau)=g(s,\tau)+h(s,\tau),$  where $g\in bAA(\mathbb{T}\times\mathbb{T},\mathbb{Y})$ and $h\in C_{rd_0}(\mathbb{T}^+\times\mathbb{T}^+,\mathbb{Y}).$
	\begin{equation}\label{eq-3.27}
		\Phi(s):=\int\limits_{-\infty}^{s}{\mathcal{H}(s,\tau,y(\tau))}\Delta \tau=\int\limits_{-\infty}^{s}{\mathcal{J}(s,\tau,y(\tau))}\Delta \tau+\int\limits_{-\infty}^{s}{\mathcal{K}(s,\tau,y(\tau))}\Delta \tau.
	\end{equation}
	Since every asymptotically almost automorphic function is also almost automorphic. By Proposition \ref{Prop-3.6}, we have 
	\begin{equation}\label{eq-3.28}
		\Phi_3(s):=\int\limits_{-\infty}^{s}{\mathcal{J}(s,\tau,y(\tau))\Delta \tau}\in AA(\mathbb{T},\mathbb{Y}).
	\end{equation}
	Also by the given condition as in  $(H_2^{'}),$ we have 
	\begin{equation}\label{eq-3.29}
		\lim\limits_{s\to\infty}\int\limits_{-\infty}^{s}{\mathcal{K}(s,\tau,y(\tau))}\Delta \tau=0.
	\end{equation}
	By the help \eqref{eq-3.28} and \eqref{eq-3.29}, we can conclude from \eqref{eq-3.27} that $\Phi\in AAA(\mathbb{T}^+,\mathbb{Y}).$
\end{proof}
\begin{prop}\label{prop-8}
	Under the hypothesis (H$_1^{'})-$ (H$_3^{'}),$ the function $F_3$ defined as $$\Psi(s);=\int\limits_{-\infty}^{s}{T(s-\sigma(t))\mathcal{F}\left(t,y(t),\int\limits_{-\infty}^{t}{\mathcal{H}(t,\tau,y(\tau))}\Delta\tau\right)}\Delta t$$ is also aysmptotically almost automorphic. 
\end{prop} 
Since $\mathcal{F}\in AAA(\mathbb{T}^+\times\mathbb{Y},\mathbb{Y},\mathbb{Y}),$ we can easily show that the function $\Psi\in AAA(\mathbb{T},\mathbb{Y}).$ \\
For reference of proof we refer \textbf{Proposition 3.6} in \cite{Milce} and \textbf{Lemma 3.3} and \textbf{Lemma 3.4} of \cite{Cao}.\\
Now we establish our main result of this section.
\begin{thm}
	Under the given hypothesis $(H_1^{'})-(H_4^{'})$ the integral-dynamic equation given by \eqref{eq-3.1}-\eqref{eq-3.2} admits a unique solution which is also asymptotically almost automorphic, provided  $0<M_\mathcal{F}=\sup\big\{\|\mathcal{F}(s,0,z)\|_{\mathbb Y};~s\in \mathbb{T},~z \in \mathbb Y \big\}<\infty.$
\end{thm}
Let $\Upsilon:AAA(\mathbb{T}^+,\mathbb{Y})\to AAA(\mathbb{T}^+,\mathbb{Y}),$ given by  $$\Upsilon(y)(s)=\int\limits_{-\infty}^{s}{T(s-\sigma(t)}\mathcal{F}\left(t,y(t),\int\limits_{-\infty}^{t}{\mathcal{H}(t,\tau,y(\tau))\Delta \tau}\right)\Delta t,~ \forall y\in AAA((\mathbb{Y},\mathbb{X})).$$
We show that $\Upsilon,$ defined as above has a unique fixed point. For the same, we follow the following steps:\\

\textbf{Step:1} $\Upsilon$ is well defined. Let $y\in AAA(\mathbb{T}^+,\mathbb{Y}).$ Then using the hypothesis $(H_2^{'}),$ by \textbf{Proposition \ref{prop-7}}, the function $\Phi(s):=\int\limits_{-\infty}^{s}{\mathcal{H}(s,\tau,y(\tau))}\Delta \tau \in AAA(\mathbb{T}^+,\mathbb{Y}).$ And hence by \textbf{Proposition \ref{prop-5}}, the function $\Gamma(s):=\mathcal{F}\left(s,y(s),\int\limits_{-\infty}^{s}{\mathcal{H}(s,\tau,y(\tau))}\Delta \tau\right)\in AAA(\mathbb{T}^+,\mathbb{Y}).$ Then by using \textbf{Proposition} \ref{prop-8}, the function  $\Psi(s):=\int\limits_{-\infty}^{s}{T(s-\sigma(t))\mathcal{F}\left(t,y(t),\int\limits_{-\infty}^{t}{\mathcal{H}(t,\tau,y(\tau))}\Delta\tau\right)}\Delta t \in AAA(\mathbb{T}^+,\mathbb{Y}).$ Hence $\Upsilon(y)(s)\in AAA(\mathbb{T}^+,\mathbb{Y}).$ This concludes the first step.  \\

\textbf{Step:2} By $(H_4^{'}),$ there exists $r>0$ such that \begin{equation}\label{eq-3.30}
	\frac{\alpha r}{M}-r(1+\tilde{\mu}\alpha)\left(L_{\mathcal{F}}^*+{L_{\mathcal{H}}^{1}}^*\right)>(1+\tilde{\mu}\alpha)M_F.
\end{equation} 
Let us consider the set, $\mathcal{B}_r=\left\{y \in AAA(\mathbb{T}^+,\mathbb{Y}): \|y(s)\|\leq r \right\} \subset AAA(\mathbb{T}^+,\mathbb{Y}).$ 
For $y\in\mathcal{B}_r,$ let us assume that $z(t)=\int\limits_{-\infty}^{t}{\mathcal{H}(t,\tau,y(\tau))}\Delta\tau$ and $z_0(t)=\int\limits_{-\infty}^{t}{\mathcal{H}(t,\tau,0)}\Delta\tau$ such that 
\begin{eqnarray}
	\|z(t)-z_0(t)\|&=&\|\int\limits_{-\infty}^{t}{\mathcal{H}(t,\tau,y(\tau))}\Delta\tau-\int\limits_{-\infty}^{t}{\mathcal{H}(t,\tau,0)}\Delta\tau\|\notag\\
	&\leq& \int\limits_{-\infty}^{t}\|{\mathcal{H}(t,\tau,y(\tau))}-{\mathcal{H}(t,\tau,0)}\|\Delta\tau\notag\\
	&\leq& \int\limits_{-\infty}^{t}{L_{\mathcal{H}}(\tau)}\|y(\tau)\|\Delta\tau.
\end{eqnarray}
Then we have,
\begin{eqnarray}
	\|\Upsilon(y)(t)\|&=&\left\|\int\limits_{-\infty}^{s}{T(s-\sigma(t))\mathcal{F}\left(t,y(t),z(t))\right)}\Delta t\right\|\notag\\
	&=&\left\|\int\limits_{-\infty}^{s}{T(s-\sigma(t))\left[\mathcal{F}(t,y(t),z(t)-\mathcal{F}(t,0,z_0(t)\right]}\Delta t+\int\limits_{-\infty}^{s}{T(s-\sigma(t))\mathcal{F}(t,0,z_0(t)}\Delta t\right\|\notag\\
	&=&\int\limits_{-\infty}^{s}\left\|{T(s-\sigma(t))\left[\mathcal{F}(t,y(t),z(t)-\mathcal{F}(t,0,z_0(t)\right]}\right\|\Delta t+\int\limits_{-\infty}^{s}\left\|{T(s-\sigma(t))\mathcal{F}(t,0,z_0(t)}\right\|\Delta t\notag\\
	&\leq&\int\limits_{-\infty}^{s}M e_{\alpha}(s-\sigma(t)) \left\{L_\mathcal{F}(t)(\|y(t)\|+{L_{\mathcal{H}}^1}(t)\|y\|)\right\}\Delta t +M M_{\mathcal{F}}\int\limits_{-\infty}^{s} e_{\alpha}(s-\sigma(t))\notag\\
	&\leq& M(L_{\mathcal{F}}^*+{L_{\mathcal{H}}^{1}}^*)r\left(\frac{1+\Tilde{\mu}\alpha}{\alpha}\right)+M M_{\mathcal{F}}\left(\frac{1+\Tilde{\mu}\alpha}{\alpha}\right)\notag\\
	&=&M \left(\frac{1+\Tilde{\mu}\alpha}{\alpha}\right)\left((L_{\mathcal{F}}^*+{L_{\mathcal{H}}^{1}}^*)r+M_{\mathcal{F}}\right)\notag\\
	&<& r,
\end{eqnarray}
which proves that $\Upsilon:= \mathcal{B}_r\to\mathcal{B}_r.$\\
Also from \eqref{eq-3.30}, we have 
\begin{eqnarray}\label{eq-3.33}
	\frac{\alpha r}{M}&-&r(1+\Tilde{\mu}\alpha)\left(L_{\mathcal{F}}^*+{L_{\mathcal{H}}^{1}}^*\right)>0\notag\\
	\Rightarrow \gamma&=&M (L_{\mathcal{F}}^*+{L_{\mathcal{H}}^{1}}^*)\left(\frac{1+\Tilde{\mu}\alpha}{\alpha}\right)<1 
\end{eqnarray}
\textbf{Step:3}  Now for $y_1,y_2 \in \mathcal{B}_r,$ let $z_1(t)=\int\limits_{-\infty}^{t}{\mathcal{H}(t,\tau,y_1(\tau))}\Delta\tau$ and $z_2(t)=\int\limits_{-\infty}^{t}{\mathcal{H}(t,\tau,y_2(\tau))}\Delta\tau.$ Then we get
\begin{eqnarray}
	\|z_2(t)-z_1(t)|&=&\|\int\limits_{-\infty}^{t}{\mathcal{H}(t,\tau,y_2(\tau))}\Delta\tau-\int\limits_{-\infty}^{t}{\mathcal{H}(t,\tau,y_1(\tau))}\Delta\tau\|\notag\\
	&\leq& \int\limits_{-\infty}^{t}\|{\mathcal{H}(t,\tau,y_2(\tau))}-{\mathcal{H}(t,\tau,y_1(\tau))}\|\Delta\tau\notag\\
	&\leq& \int\limits_{-\infty}^{t}{L_{\mathcal{H}}(\tau)}\|y_2(\tau)-y_1(\tau)\|\Delta\tau.
\end{eqnarray}
Now
\begin{eqnarray*}
	\|\Upsilon(y_2)(t)-\Upsilon(y_1)(t)\|&=&\left\|\int\limits_{-\infty}^{s}{T(s-\sigma(t))\mathcal{F}\left(t,y_2(t),z_2(t)\right)}\Delta t\right.\\
	&-&\left.\int\limits_{-\infty}^{s}{T(s-\sigma(t))\mathcal{F}\left(t,y_1(t),z_1(t)\right)}\Delta t\right\|\\
	&\leq& \int\limits_{-\infty}^{s}\|T(s-\sigma(t))\| \|\mathcal{F}\left(t,y_2(t),z_2(t)\right)-\mathcal{F}\left(t,y_1(t),z_1(t)\right)\|\Delta t\\
	&\leq& \int\limits_{-\infty}^{s}M e_{\alpha}(s-\sigma(t)) \left\{L_\mathcal{F}(t)(\|y_2(t)-y_1(t)\|+\|z_2(t)-z_1(t)\|)\right\}\Delta t\\
	&\leq& \int\limits_{-\infty}^{s}M e_{\alpha}(s-\sigma(t)) \left\{L_\mathcal{F}(t)(\|y_2(t)-y_1(t)\|+\int\limits_{-\infty}^{t}{L_{\mathcal{H}}(\tau)}\|y_2(\tau)-y_1(\tau)\|\Delta\tau)\right\}\Delta t\\
	&\leq& \int\limits_{-\infty}^{s}M e_{\alpha}(s-\sigma(t)) \left\{L_\mathcal{F}(t)(\|y_2(t)-y_1(t)\|+{L_{\mathcal{H}}^{1}}(t)\|y_2-y_1\|)\right\}\Delta t\\
	&\leq& M (L_{\mathcal{F}}^*+{L_{\mathcal{H}}^{1}}^*)\|y_2-y_1\|\int\limits_{-\infty}^{t}{e_{\alpha}(s-\sigma(t))}\\
	&\leq& M (L_{\mathcal{F}}^*+{L_{\mathcal{H}}^{1}}^*)\left(\frac{1+\Tilde{\mu}\alpha}{\alpha}\right)\|y_2-y_1\|\\
	&\leq& \gamma \|y_2-y_1\|~(\text{using equation \eqref{eq-3.33}})
\end{eqnarray*}
Thus we have established that the function $\Upsilon$ is a contraction and thus by the Banach Contraction Principle, there exists a unique $y\in AAA(\mathbb{T}^+,\mathbb{Y})$ for which $\Upsilon y=y.$
\begin{example}
Consider the time scale, $$\mathbb{P}_{a,b}=\bigcup\limits_{k=0}^{\infty}[k(a+b),k(a+b)+a].$$
This time scale is invariant under translation and contains $0.$ This time scale is one of the most useful time scale which is used to model population dynamics of certain species with certain life span, whose measurements are given in terms of $a$ and $b.$\\
Let us now consider the dynamic equation 
\begin{eqnarray}
	y^\Delta(s)&=&Ay(s)+\mathcal{F}\left(s,y(s),\int\limits_{s_0}^{s}\mathcal{H}\left(s,\tau,y(\tau)\right)\Delta\tau\right)\\
	y(s_0)&=&y_0,
\end{eqnarray}
on the time scale $\mathbb{P}_{a,b},$ where $A$ is some generator of a exponentially stable $C_0-$semigroup, $\{T(s):s\in \mathbb{T}\}$ such that $\|T(s-s_0)\|\leq M e_{\ominus \alpha}(s-s_0).$\\ 
We take $s_0=0$ and $S=2m+1,$ for some $m\in \mathbb{N}.$\\
Let us take,\\
 $\mathcal{F}(s,x,y)=\alpha_1\sin\left(\frac{1}{2+\cos s+\cos\sqrt{2}s}\right)[\sin x+y]+\alpha_2 e_{\ominus\alpha}{(s,s_0)},$ where $\alpha>0, \alpha_1\in\left(0,\frac{1}{2M(2m+1)(1+2(2m+1))}\right), \alpha_2$ are some constant and
 $\mathcal{H}(s,t,y)=\sin s\cos t+\sin y+\cos y.$\\
 At the first instance we note that $\mathcal{F},$ given above is an asymptotically almost automorphic function, where $\mathcal{G},$ \cite{Cao} given by $\mathcal{G}(s,x,y)=\sin\left(\frac{1}{2+\cos s+\cos\sqrt{2}s}\right)[\sin x+y]\in AA(\mathbb{T}\times\mathbb{Y}\times\mathbb{Y},\mathbb{Y}).$ Also $\lim\limits_{s\to\infty}e_{\ominus\alpha}{(s,s_0)}=0 ~\text{(by \textbf{Lemma \ref{lemma-2.1}})}.$\\
We can also verify that $\mathcal{F}$ satisfies Lipschitz condition given by $(H_1^{'}),$ as
\begin{eqnarray}
	\big{\|}\mathcal{F}(s,x_1,y_1)-\mathcal{F}(s,x_2,y_2)\big{\|}_2^2&=&\int\limits_{0}^{\pi}\alpha_1^2\Bigg|sin\left(\frac{1}{2+\cos t+\cos\sqrt{2}t}\right)[\sin x_1+y_1]+e_{\ominus}{(t,s_0)}\notag\\
		&-&\sin\left(\frac{1}{2+\cos t+\cos\sqrt{2}t}\right)[\sin x_2+y_2]+e_{\ominus}{(t,s_0)}\Bigg|^2\Delta t\notag\\
	\Rightarrow\big{\|}\mathcal{F}(s,x_1,y_1)-\mathcal{F}(s,x_2,y_2)\big{\|}_2^2&\leq&\alpha_1^2\Bigg|\sin\left(\frac{1}{2+\cos t+\cos\sqrt{2}t}\right)\Bigg|^2 |[\sin x_1+y_1]-[\sin x_2+y_2]|^2\notag\\
	&\leq&\alpha_1^2\|x_1-x_2\|_2^2+\|y_1-y-2\|_2^2,~ \text{for some $1\geq c\in\mathbb{R}.$}\notag
\end{eqnarray}
i.e., $\big{\|}\mathcal{F}(s,x_1,y_1)-\mathcal{F}(s,x_2,y_2)\big{\|}_2\leq \alpha_1\|x_1-x_2\|_2+\|y_1-y-2\|_2$\\
From the above equation we can verify that, $\mathcal{F}$ satisfies the Lipschitz condition given by $H_1$ as well as $(H_1^{'}).$
Furthermore, $\mathcal{H}\in bAA(\mathbb{T}\times\mathbb{T}\times\mathbb{Y},\mathbb{Y})\Rightarrow\mathcal{H}\in bAAA(\mathbb{T}\times\mathbb{T}\times\mathbb{Y},\mathbb{Y}).$ \\
Also
\begin{eqnarray}
	\|\mathcal{H}(s,t,y_1)-\mathcal{H}(s,t,y_2)\|_2&=&\|\sin s\cos t+\sin y_1+\cos y_1-\sin s\cos t+\sin y_2+\cos y_2\|_2\notag\\
	&\leq&\|\sin y_1-\sin y_2\|_2+\|\cos y_1-\cos y_2\|_2\notag\\
	&\leq& 2\|y_1-y_2\|_2.
	\end{eqnarray}
	Therefore $H_2$ as well as $(H_2^{'})$ is also satisfied. 
Also $H_3$ which is also same as $(H_3^{'})$ is evident by assumptions on $A$.\\
	Now $M(S-s))L_{\mathcal{F}}^*(1+L_{\mathcal{H}}(S-s_0))=M(2m+1)\alpha_1(1+2(2m+1))<1,$ which verifies the hypothesis $H_4.$ 
	Hence, theorem \ref{thm-3.1} ensures us a unique solution to the given equation.
\end{example}


\begin{thebibliography}{00}
	\bibitem{Abbas}Abbas, S.: Qualitative analysis of dynamic equations on time scales. Electron. J. Diff. Equ. 2018. 1--13 (2017).
	\bibitem{Atici} Atici, F.M., Biles, D.C., Lebedinsky, A.: An application of time scales to economics. Math. Comput. Model. 43(7–8), 718--726 (2006).
	\bibitem{Uysal}Atici, F.M., Uysal, F.: A production-inventory model of HMMS on time scales. Appl. Math. Letters. 21(3) 236--243 (2008).
	\bibitem{Bellman}Bellman, R.:  The stability of solutions of linear differential equations. Duke Math. J. 10(4), 643--647 (1943).
	
	\bibitem{Bihari} Bihari, I.: A generalization of a lemma of Bellman and its application to uniqueness problems of differential equations. Acta Math. Hungarica. 7(1), 81--94 (1956).
	\bibitem{Bohner2} Bohner, M., Peterson, A.: Dynamic equations on time scales. An introduction with applications. Birkhäuser Boston, Inc., Boston, MA (2001).
	\bibitem{Bohner1}Bohner, M., Peterson, A.: Advances in Dynamic Equations on Time Scales. Birkh\"auser Boston Inc., Boston (2003).
	\bibitem{BohnerandTikare}Bohner, M., Tikare, S., dos Santos, I.L.D.: First order nonlinear dynamic initial value problems. Int. J. Dyn. Syst. Differ. Equ. 11(3-4), 241--254 (2021). 
	\bibitem{Bohner 3}Bohner, M., Pallabi S. Scindia, Tikare, S. Qualitative Results for Nonlinear Integro-Dynamic Equations via Integral Inequalities. Qualitative Theory of Dynamical Systems (2022) 21:106
	\bibitem{Cao} Cao, J., Huang, Z., N'gu\'er\'ekata, G.M. (2018). Existence of asymptotically almoxt automorphic mild solutions for nonautonomous semilinear evalution equations, Electronic Journal of Differential Equations,   37, 1--16 (2018).
	\bibitem {Chicon} Cicho\'n, M., Kubiaczyk, I., Sikorska-Nowak, A., Yantir, A.: Existence of solutions of the dynamic Cauchy problem in Banach spaces. Demo. Math. 45(3), 561--573 (2012).
        \bibitem{Coppel}Coppel, W.A. Dichotomies in stability theory. Lecture Notes in Mathematics. Springer, Berlin, 1978.
       \bibitem{Cosme} Duque, C.,  Leiva, H.,  Gallo, R., Tridane. On the Existence and Stability of Bounded Solutions for Abstract Dynamic Equations on Time Scales. International Journal of Differential Equations. (2023). 10.1155/2023/8489196. 
 
	\bibitem{Guerekata 1}Mophou,G., N'gu\'er\'ekata, G. M., and Milc\'e, A. Almost automorphic functions of order n and application to dynamic equation on time scales. Discrete Dynamics in Nature and Society. Volume 2014, Article ID 410210, 13 pages.
	http://dx.doi.org/10.1155/2014/410210
	\bibitem{Gronwall} Gr\"onwall, T.H.: Note on the derivatives with respect to a parameter of the solutions of a system of differential equations, Ann. Math. 20(2), 293--296 (1919).
	\bibitem{Hamza and Oraby} Hamza, A.\& Oraby, M. Stability of abstract dynamic equations on time scales. \emph{Advances in Difference Equations.} 2012. 10.1186/1687-1847-2012-143.
	\bibitem{Hamza and Oraby 1} Hamza, A.\& Oraby, M. Stability of abstract dynamic equations on time scales by Lyapunov's second method. \emph{Turk J Math} (2018) 42: 841-861
	\bibitem{HA} Hilger, S.: Analysis on measure chains, a unified approach to continuous and discrete calculus. Results Math. 18, 18--56 (990).
	\bibitem{HD} Hilger, S.: Differential and difference calculus, unified. Nonlinear Anal. 30, 2683--2694 (1997).
	\bibitem{Karpuz} Karpuz, B.: On the existence and uniqueness of solutions to dynamic equations. Turk. J. Math. 42(3), 1072--1089 (2018).
	\bibitem{Kaufmann} Kaufmann, E.R.: A Kolmogorov predator-prey system on a time scale. Dyn. Systems Appl. 23(4), 561--573 (2014).
\bibitem{Kere} Kéré, M., N’Guérékata, G.: Almost Automorphic Dynamic Systems on Time Scales. Panamerican Mathematical Journal. 28. 19-37.(2018).
	\bibitem{Kubiaczyk} Kubiaczyk, I., Sikorska-Nowak, A.: Existence of solutions of the dynamic Cauchy problem on infinite time scale intervals. Discuss. Math. Differ. Incl. Control Optim. 29, 113--126 (2009).
	\bibitem{LaSalle} LaSalle, J.: Uniqueness theorems and successive approximations. Annal. Math. 50(3), 722--730 (1949).
	\bibitem{Lizama and Mesquita} Lizama, C., Mesquita, J.G. Almost Automorphic Solutions of Dynamic Equations on Time Scales. \emph{Journal of Functional Analysis.} 265,2267-2311 (2013). https://doi.org/10.1016/j.jfa.2013.06.013
        \bibitem{Li and Wang} Li, Y., Wang, C.: Almost periodic functions on time scales and its applications. Discrete Dynamics in Nature and
	Society. vol. 2011, Article ID 727068, 20 pages, 2011.
        \bibitem{Li&Wang} Li, Y., Wang, C.: Pseudo almost periodic functions and pseudo almost periodic solutions to dynamic equations on times scales. Adv. Difference Equ.  77, 24 pp (2012)
	\bibitem{Lizama and Mesquita 1} Lizama, C., Mesquita, J.G.  (2019). Asymptotically almost automorphic solutions of dynamic equations on time scales. Topological Methods in Nonlinear Analysis.(2019). 1.10.12775/TMNA.2019.024.
	\bibitem{Milce} Milc\'e, A. Asymptotically almost automorphic solutions for some integro-dynamic equations with nonlocal initial conditions on time scales. Dynamics of Continuous, Discrete and Impulsive Systems. Mathematical Analysis 23, 27-46, (2016).
	\bibitem{Pata} Pata, V.: Fixed point theorems and applications. Cham. Springer. 11, (2019).
	\bibitem{Pachpatte} Pachpatte, B.G.:  {\it Inequalities for differential and integral equations}. San Diego: Academic Press (1998).
	\bibitem{Santos1} Santos, I.L.D.: On qualitative and quantitative results for solutions to first-order dynamic equations on time scales. Bol. Soc. Mat. Mex. (3) 21(2), 205--218 (2015).
	\bibitem{Shen} Shen, Y.: The Ulam stability of first order linear dynamic equations on time scales. Results Math.
	72(4), 1881--1895 (2017).
	\bibitem{Tisdell} Tisdell, C.C., Zaidi, A.: Basic qualitative and quantitative results for solutions to nonlinear, dynamic
	equations on time scales with an application to economic modelling. Nonlinear Anal. 68(11), 3504--3524 (2008).
        \bibitem{Wong and Others} Wong, F.-H., Yeh, C.-C., Hong, C.-H.:Gronwall inequalities on time scales. Math.Inequal.Appl.9(1),75-86 (2006). 10.7153/mia-09-08.
	\bibitem{Xiao Zhu Liang} Xiao, T., Zhu, X., Liang, J. Pseudo- 
    almost automorphic mild solutions to nonautonomous differential equations and applications. Nonlinear Anal. Theory Methods Appl. 70, 4079-4085 (2019).
	
	\bibitem{Zhu} Zhu, Y., Jia, G.: Linear feedback of mean-field stochastic linear quadratic optimal control problems on time scales. Math. Probl. Eng.  Art. ID 8051918, 11 (2020).
	\bibitem{Zhuang} Zhuang, K.: Periodic solutions for a stage-structure ecological model on time scales. Electron. J. Diff. Equ.  88, 7 (2007).  
 \bibitem{Zhu and Wang} Zhu, Z.-Q., Wang, Q.-R.: Existence of nonoscillatory solutions to neutral dynamic equations on time scales. J. Math. Anal. Appl. 335(2), 751–762 (2007)
\end{thebibliography}
\end{document}